\documentclass[11pt,reqno,twoside]{article}

\usepackage{cmap} 

\usepackage[T1]{fontenc}
\usepackage[utf8]{inputenc}
\usepackage{graphicx}
\usepackage{placeins}
\usepackage{enumerate}
\usepackage{algcompatible}

\usepackage{verbatim}
\newcommand{\comments}[1]{}

\usepackage{soul}

\usepackage{setspace}

\let\counterwithin\relax  
\usepackage{lmodern} 

\usepackage{bm} 

\usepackage{bbold}

\usepackage{amsmath,amsbsy,amsgen,amscd,amsthm,amsfonts,amssymb} 

\usepackage[centering,top=1in,bottom=1in,left=1in,right=1in]{geometry}

\usepackage[sf,bf,compact]{titlesec}

\usepackage[usenames,dvipsnames]{xcolor}

\definecolor{dark-gray}{gray}{0.3}
\definecolor{dkgray}{rgb}{.4,.4,.4}
\definecolor{dkblue}{rgb}{0,0,.5}
\definecolor{medblue}{rgb}{0,0,.75}
\definecolor{rust}{rgb}{0.5,0.1,0.1}

\usepackage{url}
\usepackage[colorlinks=true]{hyperref}
\hypersetup{linkcolor=dkblue}    
\hypersetup{citecolor=rust}      
\hypersetup{urlcolor=rust}     

\usepackage[final]{microtype} 

\newtheoremstyle{myThm} 
    {\topsep}                    
    {\topsep}                    
    {\itshape}                   
    {}                           
    {\sffamily\bfseries}                   
    {.}                          
    {.5em}                       
    {}  

\newtheoremstyle{myRem} 
    {\topsep}                    
    {\topsep}                    
    {}                   
    {}                           
    {\sffamily}                   
    {.}                          
    {.5em}                       
    {}  

\newtheoremstyle{myDef} 
    {\topsep}                    
    {\topsep}                    
    {\itshape}                   
    {}                           
    {\sffamily\bfseries}                   
    {.}                          
    {.5em}                       
    {}  

    \newtheoremstyle{myDef2} 
    {\topsep}                    
    {\topsep}                    
    {}                   
    {}                           
    {\sffamily\bfseries}                   
    {.}                          
    {.5em}                       
    {}  

\theoremstyle{myThm}
\newtheorem{theorem}{Theorem}[section]
\newtheorem{lemma}[theorem]{Lemma}

\newtheorem{corollary}[theorem]{Corollary}

\theoremstyle{myDef}
\newtheorem{definition}[theorem]{Definition}

\theoremstyle{myDef2}

  \newtheorem{examplex}[theorem]{Example}
 \newenvironment{example}
  {\pushQED{\qed}\examplex}
  {\popQED\endexamplex}

      \newtheorem{propertyx}[theorem]{Property}
 \newenvironment{property}
  {\pushQED{\qed}\propertyx}
  {\popQED\endpropertyx}

\theoremstyle{myRem}

\usepackage{fancyhdr}
\let\originalleft\left
\let\originalright\right
\renewcommand{\left}{\mathopen{}\mathclose\bgroup\originalleft}
\renewcommand{\right}{\aftergroup\egroup\originalright}

\usepackage{mathtools}
\mathtoolsset{centercolon}  

\renewcommand{\phi}{\varphi}
\newcommand{\eps}{\varepsilon}

\providecommand{\mathbbm}{\mathbb} 

\newcommand{\R}{\mathbbm{R}}

\newcommand{\C}{\mathbbm{C}}

\newcommand{\G}{\mathcal{G}}
\newcommand{\F}{\mathcal{F}}

\definecolor{mygreen}{rgb}{0.13,0.55,0.13}

\newcommand{\M}{\mathcal{M}}

\newcommand{\Nc}{\mathcal{N}}

\newcommand{\J}{{\mathsf{J}}}

\usepackage[font = small, margin=30pt]{caption}

\usepackage[]{algorithm}
\usepackage{enumerate}

\usepackage{graphicx}

\usepackage{authblk}
\usepackage{chngcntr}
\usepackage{mathrsfs}
\counterwithin{table}{section}
\counterwithin{algorithm}{section}

\usepackage{tikz}
\usetikzlibrary{shapes,arrows,decorations.pathmorphing,backgrounds,positioning,fit,petri}
\usetikzlibrary{calc}
\usetikzlibrary{matrix}
\usetikzlibrary{backgrounds}
\usetikzlibrary{shapes.geometric}
\usepackage{pgfplots}
\pgfplotsset{compat=newest}
\usepgfplotslibrary{groupplots}
\pgfdeclarelayer{background}
\pgfdeclarelayer{bBack}
\pgfsetlayers{bBack,background,main}
\usetikzlibrary{fit}

\title{Hierarchical Bayesian Inverse Problems:\\ A High-Dimensional Statistics Viewpoint} 

\author{D. Sanz-Alonso and N. Waniorek}

\date{University of Chicago}

\vspace{-.5in}

\makeatletter\@addtoreset{section}{part}\makeatother%
\numberwithin{equation}{section}

\newcommand{\upperRomannumeral}[1]{\uppercase\expandafter{\romannumeral#1}}

\DeclareMathOperator*{\argmax}{arg\,max}
\DeclareMathOperator*{\argmin}{arg\,min}

\begin{document}
\maketitle 


\abstract{     
This paper analyzes hierarchical Bayesian inverse problems using techniques from high-dimensional statistics.
Our analysis leverages a property of hierarchical Bayesian regularizers that we call \emph{approximate decomposability} to
obtain non-asymptotic bounds on the reconstruction error attained by \emph{maximum a posteriori} estimators. 
The new theory explains how hierarchical Bayesian models that exploit sparsity, group sparsity, and sparse representations of the unknown parameter can achieve accurate reconstructions in high-dimensional settings.
}\\




\section{Introduction}\label{sec:introduction}
The problem of estimating a high-dimensional parameter from a few noisy and indirect measurements is of central importance in science and engineering \cite{donoho2000high,calvetti2023bayesian,sanz2023inverse} and arises, for instance, in genomics \cite{bickel2009overview}, spatial statistics \cite{gelfand2010handbook},  imaging \cite{david2002hyperspectral,lustig2008compressed}, and data assimilation \cite{law2015data}.
High-dimensionality ---which in many inverse problems results from the unknown parameter being a fine discretization of a function \cite{bui2013computational,stuart2010inverse}--- presents both statistical and computational challenges. From a statistical perspective, consistent estimation of a $d$-dimensional unknown from $n\ll d$ measurements is only possible when the problem exhibits some underlying low-dimensional structure \cite{buhlmann2011statistics,wainwright2019high}. 
From a computational perspective, the scalability of optimization and sampling algorithms hinges on adequately exploiting such structure \cite{cotter2013mcmc,sanz2023analysis,trillos2020consistency,agapiou2017importance}. 

In the Bayesian approach to inverse problems, low-dimensional structure can be imposed on the reconstruction through the choice of prior distribution \cite{sanz2023inverse,kaipio2006statistical}.
Hierarchical modelling \cite{gelman1995bayesian} provides an expressive approach to prior design. The idea is to model the parameters in the prior as random variables distributed according to a hyperprior, thus specifying a joint prior distribution for the unknown of interest and the prior parameters. Conditioning on the observed data yields the posterior distribution, which represents the Bayesian solution to the inverse problem \cite{stuart2010inverse}.
We consider a popular family of sparsity-promoting hierarchical models with a mean-zero Gaussian prior and gamma hyperpriors on the prior variances. This model was originally proposed for inverse problems in medical imaging, such as cerebral source localization \cite{calvetti2009conditionally} and positron emission tomography \cite{bardsley2010hierarchical}. Heuristically, sparsity is promoted in the following way: the gamma hyperprior imposes the belief that most of the variances are small except for a few outliers; consequently, the entries of prior draws will be nearly zero except for those with outlier variances. Beyond the original motivation in imaging, similar hierarchical models have been applied to deconvolution problems \cite{calvetti2020sparse}, dictionary learning \cite{pragliola2022overcomplete,waniorek2023bayesian,bocchinfuso2023bayesian}, multiple measurement inverse problems \cite{glaubitz2023leveraging}, adaptive meshing \cite{bocchinfuso2023adaptive}, and identification of dynamical systems \cite{agrawal2021variational}.

Hierarchical Bayesian models with gamma hyperpriors have been extensively studied from a computational viewpoint. For linear inverse problems, the posterior mode ---known as the \emph{maximum a posteriori} (MAP) estimator---  can be efficiently computed using an iterative alternating scheme that converges linearly \cite{calvetti2015hierarchical}. Path-following methods for MAP estimation were developed in \cite{si2022path}.
For nonlinear inverse problems, the MAP estimator can be approximated using ensemble Kalman methods \cite{kim2023hierarchical}. 
Recently, algorithms for sampling \cite{calvetti2023computationally} and variational inference \cite{agrawal2021variational} have been proposed. While this body of work compellingly showcases the computational benefits afforded by these hierarchical models, reconstruction error bounds are not yet available.

In this paper, we present a new statistical analysis of a family of hierarchical Bayesian models with gamma hyperpriors, obtaining the first error bounds on the reconstruction error of MAP estimators. Our bounds are \emph{non-asymptotic}: they are applicable for any dimension and number of measurements. We view  MAP estimators as particular instances of regularized M-estimators, that is, as being defined by minimizing an objective comprising a loss term and a regularization term.
M-estimators play a central role in the statistical literature \cite{geer2000empirical}. A canonical example is the Lasso estimator \cite{tibshirani1996regression}, which  minimizes the sum of a least-squares loss and an $\ell_1$ penalty
that enforces sparsity in the reconstruction of the unknown \cite{bickel2009simultaneous}. Other widely used M-estimators include the group Lasso \cite{bach2008consistency}, graphical Lasso \cite{jacob2009group}, sparse generalized linear models \cite{van2008high}, and nuclear norm penalization for matrix completion \cite{candes2012exact,recht2010guaranteed}. The similarities between regularized M-estimation and MAP estimation are immediate, and in fact many M-estimators can be derived as MAP estimators of certain posterior distributions. For instance, the Lasso corresponds to the MAP estimator of a posterior derived from a linear forward model with additive Gaussian noise and a Laplace prior \cite{park2008bayesian}. 
Our new theory for hierarchical Bayesian models builds on the general framework for M-estimation developed in \cite{negahban2012unified}, which identifies two key properties of the loss and the penalty under which accurate reconstructions are possible in high-dimensional settings: \emph{restricted strong convexity} of the loss and \emph{decomposability} of the regularizer. As we shall see, decomposability does not hold for the penalties that stem from the hierarchical Bayesian models for inverse problems that we consider. To tackle this challenge, we will introduce a weaker notion of \emph{approximate decomposability} under which we obtain non-asymptotic bounds on the deviation between the MAP estimator and the true unknown generating the data.


\subsection*{Outline and Main Contributions}
Section \ref{sec:Problem Formulation} formalizes the inverse problem of interest and introduces three guiding examples of hierarchical Bayesian models that promote sparsity, group sparsity, and sparse representations. These examples will showcase the broad applicability of our general theory, and will also contribute to making the presentation accessible to a wide audience.
    Section \ref{sec:GeneralTheory} sets forth our unified theoretical framework to establish reconstruction bounds leveraging approximate decomposability of hierarchical Bayesian regularizers. 
    Section \ref{sec:convergence} applies the general theory to our three guiding examples.
    To streamline the presentation, we defer proofs to Section \ref{sec:proofs} and several appendices. We close in Section \ref{sec:conclusions} with conclusions and questions for future research. 

    \medskip
    
This paper contributes to bridge the fields of inverse problems and high-dimensional statistics.
 In addition to providing a unified and accessible perspective of recent literature on inverse problems and statistical M-estimation theory, this paper makes the following original contributions:
\begin{itemize}
    \item Theorem \ref{thm:GSIAS computational properties} proves linear convergence of an iterative alternating scheme to compute the MAP estimator for a hierarchical model that promotes group sparsity. 
    Example \ref{ex:OIAS Example} proposes a novel hierarchical model to promote sparse representations in an overcomplete dictionary. We derive an iterative alternating scheme to compute the MAP estimator, and we prove its linear convergence in Theorem \ref{thm:OIAS computational properties}.
    \item Definition \ref{def:approximate decomposability} introduces the notion of approximate decomposability, and Lemma \ref{lemma:decomposabilityproperty} highlights an important consequence of this property. Our main result, Theorem \ref{thm:reconstruction error theorem}, establishes reconstruction error bounds for M-estimators with approximately decomposable regularizers. Our theory generalizes the one in \cite{negahban2012unified}, which only holds for exactly decomposable regularizers and is not applicable to the hierarchical models we consider.
    \item Theorems \ref{thm:deterministic error bound IAS} and \ref{thm:deterministic bound IAS LQ} provide 
    the first known reconstruction error bounds for MAP estimation using sparsity-promoting conditionally Gaussian models with gamma hyperpriors under both well-specified and misspecified settings. For Gaussian noise models and forward maps with Gaussian columns, we sharply characterize the reconstruction errors under well-specified and misspecified settings in Corollaries \ref{thm:Gaussian Corollary}  and \ref{thm:Gaussian Corollary IAS LQ}. Our results show that with an appropriate choice of hyperparameters, the MAP estimator achieves the minimax rate for high-dimensional sparse linear regression. Analogous results for group sparsity are established in  
    Theorem \ref{thm: deterministic error bound GSIAS} and Corollary \ref{thm: Gaussian Corollary GS},
    and for sparse representations in Theorem \ref{thm: deterministic error bound OSIAS} and Corollary \ref{thm:Gaussian Corollary SR}.
\end{itemize}

\section{Problem Formulation}\label{sec:Problem Formulation}
Consider the inverse problem of reconstructing an unknown $u \in \R^d$ from data $y \in \R^n$ related by
\begin{align}\label{eq:inverse problem}
    y =A u +\eps, \quad \eps \sim \Nc(0,I).
\end{align}
Here, $A\in \R^{n\times d}$ is a given forward map and $\eps$ denotes measurement noise. We assume white noise without loss of generality ---our theory can be readily extended to correlated Gaussian noise using a whitening transformation. We are interested in high-dimensional settings where $d \gg n.$ In this underdetermined regime, accurate reconstruction hinges on low-dimensional structure that arises, for instance, when the unknown is sparse or admits a sparse representation. 

In the Bayesian approach to inverse problems, inference is based on the posterior distribution, which combines the Gaussian likelihood  $y|u \sim \Nc(Au, I)$ implied by \eqref{eq:inverse problem}
 with a prior distribution that encodes structural beliefs on the unknown. We focus on sparsity-promoting hierarchical priors, where the level of sparsity is controlled by an auxiliary parameter $\theta \in \R^k.$ Specifically, we consider priors of the form
$$\pi(u,\theta) = \pi(u|\theta) \pi_\eta(\theta),$$
where the hyperprior $\pi_\eta$ on $\theta$ depends on a parameter $\eta$ that will play an important role in the theory. 
Combining the likelihood and the hierarchical prior, we obtain the posterior distribution 
\begin{align}\label{eq:posterior}
\begin{split}
     \pi(u,\theta |y ) \propto \pi(y|u) \pi(u|\theta) \pi_\eta(\theta)  \propto \exp \bigl( -\mathsf{J}(u,\theta) \bigr),
\end{split}
\end{align} 
where $\mathsf{J}$ denotes (up to an additive constant) the negative log-density of the posterior. 
The \emph{maximum a posteriori} (MAP) estimator for $(u,\theta)$ is then given by
\begin{align}\label{eq:MAP}
    (\hat u, \hat \theta) = \argmax_{u,\theta} \pi(u,\theta | y) 
     = \argmin_{u,\theta} \mathsf{J}(u,\theta).   
\end{align}
This paper conducts a high-dimensional analysis of MAP estimators for hierarchical Bayesian models under the frequentist assumption that
the data are generated from a fixed true value $u^\star$ of the unknown, so that
\begin{align}\label{eq:frequentist inverse problem}
    y=Au^\star + \eps, \quad \eps \sim \Nc(0,I).
\end{align}
Our goal is to derive bounds on the reconstruction error $\hat{\Delta}:=\hat{u}-u^\star$ between the MAP estimator $\hat{u}$ and the true unknown parameter $u^\star$.
In a unified framework, we will analyze several hierarchical models that satisfy the following two key properties:
\begin{property}\label{property1}
 The MAP estimator $(\hat{u},\hat{\theta})$  can be efficiently computed via an alternating optimization scheme of the form:
    \begin{enumerate}
    \item  Initialize $\theta^0,$ $\ell=0.$
    \item  Iterate until convergence
    \begin{enumerate}
        \item[(i)]  $u$-update: $u^{\ell+1} :=\argmin_u \J(u, \theta^{\ell}).$  
        \item[(ii)]  $\theta$-update: $\theta^{\ell+1} :=\argmin_\theta \J(u^{\ell+1}, \theta).$ 
        \item[(iii)] $\ell\to \ell+1.$ \qedhere
    \end{enumerate}
\end{enumerate}
\end{property}
\begin{property}\label{property2}
    If $(\hat{u}, \hat{\theta})$ is the MAP estimator, then 
    $\hat{u}$ is the minimizer of an objective $\mathsf{F}: \R^d \to \R$ of the form
    \begin{equation}\label{eq:objective over u}
    \mathsf{F}(u) := \frac{1}{2n}\| y - Au \|_2^2 + \lambda \mathsf{R}_\eta(u),
\end{equation}
where $\lambda \mathsf{R}_\eta(\cdot)$ is a convex regularization term determined by the hierarchical prior. 
\end{property}
Previous work has largely focused on Property \ref{property1}, motivating hierarchical Bayesian models by the simplicity and effectiveness of alternating optimization algorithms for MAP estimation \cite{calvetti2019hierachical}. On the other hand, in this paper we focus on Property \ref{property2} and show that many regularizers $\mathsf{R}_\eta$ arising from hierarchical Bayesian models satisfy an approximate decomposability property which enables accurate reconstructions under appropriate model assumptions.
We next provide three important examples of hierarchical models that satisfy Properties \ref{property1} and \ref{property2}. In Section \ref{sec:convergence},  we will apply to these three examples the unified theory introduced in Section \ref{sec:GeneralTheory}.
\begin{example}[Hierarchical Models for Sparsity and the IAS Algorithm]\label{ex:IAS Example}
First, we consider a popular hierarchical Bayesian model with gamma hyperpriors to reconstruct sparse parameters, given by
\begin{align}\label{eq:hierarchicalmodelIAS}
\begin{split}
    u| \theta &\sim \Nc \Bigl(0, \frac{\sqrt{2}}{\lambda n} D_\theta \Bigr), \qquad D_\theta = \text{diag}(\theta),\\
    \theta_j &\stackrel{\text{i.i.d}}{\sim} \text{Gamma}\left(\frac{\sqrt{2}\theta_j^*}{\lambda n}, \frac{3}{2}+\frac{\sqrt{2}}{2}\lambda n \eta\right), \quad 1\leq j\leq d,
\end{split}    
\end{align}
with scale parameter $\frac{\sqrt{2}\theta_j^*}{\lambda n}>0$ and shape parameter $\frac{3}{2}+\frac{\sqrt{2}}{2}\lambda n \eta > 0.$ Here, the auxiliary parameter $\theta \in \R^k,$ with $k:=d,$ controls the variance of the coordinates of $u,$ and thereby the level of sparsity imposed on the reconstruction. The hierarchical Bayesian model with gamma hyperpriors in \eqref{eq:hierarchicalmodelIAS} was proposed and investigated in \cite{calvetti2015hierarchical,calvetti2019hierachical}.
From \eqref{eq:hierarchicalmodelIAS},  we obtain the joint prior 
\begin{align*}
    \pi(u,\theta)\propto \exp\left(-\frac{\lambda n}{2\sqrt{2}}\sum_{j=1}^d\frac{u_j^2}{\theta_j}-\frac{\lambda n}{\sqrt{2}}\sum_{j=1}^d \left[ \frac{\theta_j}{\theta_j^*}-\eta\log \frac{\theta_j}{\theta_j^*}\right] \right).
\end{align*}
We assume for the simplicity of our subsequent analysis that the scale parameters are chosen such that $\theta_j^*=\frac{n}{\|A_j\|_2^2},$ where $A_j\in \R^n$ denotes the $j^{th}$ column of the forward map. We then make the transformation $\left(\theta_j,u_j,A\right) \mapsto \left(\frac{\theta_j}{\theta_j^*}, \frac{u_j}{\sqrt{\theta_j^*}}, A D_{\theta^*}^{1/2} \right)$. Notice that our assumption on the scale parameters is therefore tantamount to assuming that the forward map is column-normalized so that $\frac{\|A_j\|_2}{\sqrt{n}}=1$ and that $\theta_j^*=1$ for each $1\leq j\leq d.$
As such, we consider a posterior with negative log-density given, up to an additive constant, by  
\begin{equation}\label{eq:functional}
\mathsf{J}(u,\theta) := 
\mathrlap{  \underbrace{ \phantom{  \frac{1}{2n}\| y  - Au\|_2^2 + 
 \frac{\lambda}{2\sqrt{2}} \sum_{j=1}^d\frac{u_j^2}{\theta_j} }}_{(a)}}  \frac{1}{2n}\| y  - Au\|_2^2   
+  \overbrace{  \frac{\lambda}{2\sqrt{2}} \sum_{j=1}^d\frac{u_j^2}{\theta_j} + \frac{\lambda}{\sqrt{2}} \sum_{j=1}^d\left(\theta_j-\eta \log \theta_j \right) }^{(b)} \,\, .
\end{equation}
The iterative alternating scheme (IAS) computes the MAP estimator for this objective by minimizing the least-squares objective (a) for the $u$-update and minimizing $(b)$ for the $\theta$-update. Given $(u^\ell,\theta^{\ell})$ both updates have a closed form solution:
\begin{align}\label{eq:IASalgorithm}
\begin{split}
   u^{\ell+1}& :=\left(\frac{A^TA}{n}+\frac{\lambda}{\sqrt{2}} D_{\theta^\ell}^{-1} \right)^{-1}\frac{A^Ty}{n}, \\
  \theta_j^{\ell+1} & := \frac{\eta}{2}+\sqrt{\frac{\eta^2}{4}+\frac{(u_j^{\ell+1})^2}{2}  }, \qquad 1 \le j \le d.
\end{split}
\end{align}
For large problems, the $u$-update can be efficiently computed using conjugate gradient with early stopping based on Morozov's discrepancy principle \cite{calvetti2015hierarchical,calvetti2018bayes,calvetti2019hierachical}. The closed form expression for the $\theta$-update was derived in \cite{calvetti2019hierachical}. The IAS algorithm \eqref{eq:IASalgorithm} was introduced in \cite{calvetti2015hierarchical,calvetti2019hierachical} and has been widely studied, see e.g. \cite{calvetti2015hierarchical,calvetti2018bayes,calvetti2019hierachical,calvetti2020sparsity,calvetti2020sparse}.
The following known result verifies Property \ref{property2} and summarizes other key computational properties. 
\begin{theorem}\label{thm:IASconvergence}
    For $\eta>0,$ the objective function $\J$ in \eqref{eq:functional} is strictly convex over $\R^d\times \R^d_+$ and thus admits a unique MAP estimator $(\hat{u},\hat{\theta}).$ Moreover, the $u$-iterates $u^\ell$ of the IAS algorithm \eqref{eq:IASalgorithm} converge at least linearly to $\hat{u}$ in the Mahalanobis norm $\| \cdot \|_{D_{\hat{\theta}}}.$ Furthermore, $\hat{u}$ is the global minimizer of $\mathsf{F}$ in \eqref{eq:objective over u} with 
    \begin{align}\label{eq:regularizer}
    \mathsf{R}_{\eta}(u) :=\frac{1}{\sqrt{2}}\sum_{j=1}^d\biggl[\frac{u_j^2}{2f_j(u)} + f_j(u)-\eta \log f_j(u) \biggr], \quad \quad 
    f_j(u) := \frac{\eta}{2}+\sqrt{\frac{\eta^2}{4}+\frac{u_j^2}{2}} \,.
\end{align}
\end{theorem}
As illustrated by Theorem \ref{thm:IASconvergence}, previous works focus on the convexity of $\J$ and the convergence of IAS iterates $u^\ell$ to the MAP estimator $\hat{u}.$ In this work we focus instead on understanding the convergence of the MAP estimator $\hat{u}$ to the true parameter $u^\star$ generating the data. 
\end{example}

\begin{example}[Hierarchical Models for Group Sparsity and the GS-IAS Algorithm]
\label{ex:GSIAS Example}
The second hierarchical Bayesian model we consider promotes group sparsity. A particular case of the model where each group consists of precisely three indices was proposed in \cite{calvetti2015hierarchical}, and the general setting of arbitrary group sizes was later formulated in \cite{bocchinfuso2023bayesian}.  We assume the a priori knowledge that components of the unknown $u$ are partitioned into groups $\{g_1,\ldots, g_k\}$ where $|g_j|=p_j$. The model also postulates a conditionally Gaussian prior on each group and the variance of each group is modeled via a gamma hyperprior,  
\begin{align}\label{eq:hierarchicalmodelGS-IAS}
\begin{split}
    u_{g_j}| \theta_j &\sim \Nc \Bigl(0,\frac{\sqrt{2}\theta_j}{\lambda n}C_j\Bigr), \qquad C_j \succ 0,\\
    \theta_j &\sim \text{Gamma}\left(\frac{\sqrt{2}\theta_j^*}{\lambda n}, \frac{p_j+2}{2}+\frac{\sqrt{2}}{2}\lambda n \eta\right), \quad 1\leq j\leq k,
\end{split}    
\end{align}
with scale parameter $\frac{\sqrt{2}\theta_j^*}{\lambda n}>0$ and shape parameter $\frac{p_j+2}{2}+\frac{\sqrt{2}}{2}\lambda n \eta$. Here, the auxiliary parameter $\theta\in \R^k$ controls the level of group sparsity in the reconstructed solution. From \eqref{eq:hierarchicalmodelGS-IAS}, we have the joint prior distribution
\begin{align*}
    \pi(u,\theta)\propto \exp \left(-\frac{\lambda n}{2\sqrt{2}}\sum_{j=1}^k\frac{\|u_{g_j}\|_{C_j}^2}{\theta_j}-\frac{\lambda n}{\sqrt{2}}\sum_{j=1}^k\left[\frac{\theta_j}{\theta_j^*}-\eta\log \frac{\theta_j}{\theta_j^*} \right] \right).
\end{align*}
To simplify our analysis, we assume that the scale parameters are chosen such that $\theta_j^*=\frac{n}{\|A_{g_j}\|_{C_j,2}^2}$, where 
$A_{g_j} \in \R^{n \times p_j}$ denotes the matrix obtained by keeping the columns from $A$ in group $g_j,$ and 
$\|A_{g_j}\|_{C_j,2}=\max_{\|v\|_{C_j}=1}\|A_{g_j}\|_2.$ 
Then, via the transformation $(\theta_j,u_{g_j},A_{g_j})\mapsto\left( \frac{\theta_j}{\theta_j^*},\frac{u_j}{\sqrt{\theta_j^*}},A_{g_j}\sqrt{\theta_j^*}\right)$ we see that our assumption on the scale parameters is equivalent to assuming the forward map is block-normalized so that $\frac{\|A_{g_j}\|_{C_j,2}}{\sqrt{n}}=1$ with $\theta_j^*=1.$ Applying Bayes rule, we consider a posterior with negative log-density given, up to an additive constant, by  
\begin{equation}\label{eq:GS functional}
\mathsf{J}(u,\theta) := 
\mathrlap{  \underbrace{ \phantom{  \frac{1}{2n}\| y  - Au\|_2^2 + 
\frac{\lambda}{\sqrt{2}}  \|u\|_{D_{\theta}}^2}}_{(a)}}  \frac{1}{2n}\| y  - Au\|_2^2   
+  \overbrace{  \frac{\lambda}{2\sqrt{2}} \|u\|_{D_{\theta}}^2+ \frac{\lambda}{\sqrt{2}} \sum_{j=1}^k\left(\theta_j-\eta \log \theta_j \right) }^{(b)} \,\, ,
\end{equation}
where $D_\theta$ is the block-diagonal matrix given by
\begin{align*}
    D_{\theta}:=  \begin{bmatrix}
        \theta_1C_1 & &\\
        & \ddots & \\
        & & \theta_k C_k
    \end{bmatrix}.
\end{align*}
To compute the MAP estimator for this objective, we again use an alternating optimization scheme, minimizing the least-squares objective $(a)$ for the $u$-update and minimizing $(b)$ for the $\theta$-update. Given $(u^\ell,\theta^\ell)$, these updates have closed form solutions:
\begin{align}\label{eq:GSIAS algorithm}
\begin{split}
    u^{\ell+1}&:=\left(\frac{A^TA}{n}+\frac{\lambda}{\sqrt{2}}D_{\theta^\ell}^{-1}\right)\frac{A^Ty}{n},\\
    \theta_j^{\ell+1}&:=\frac{\eta}{2}+\sqrt{\frac{\eta^2}{4}+\frac{\|u_{g_j}\|_{C_j}^2}{2}}, \qquad 1 \le j \le k.
    \end{split}
\end{align}
 The group-sparse iterative alternating scheme (GS-IAS) described by \eqref{eq:GSIAS algorithm} was derived in \cite{calvetti2015hierarchical} and it was shown that the iterates converge to the minimizer of \ref{eq:GS functional}, although no rate of convergence was derived. 
This objective and the GS-IAS algorithm admit a number of computationally convenient properties, summarized in the following theorem.
\begin{theorem}\label{thm:GSIAS computational properties}
    For $\eta>0,$ the objective function $\mathsf{J}$ in $\eqref{eq:GS functional}$ is strictly convex over $\R^d\times \R^k_+,$ and thus admits a unique MAP estimator $(\hat{u},\hat{\theta}).$ Moreover, the $u$-iterates $u^\ell$ of the GS-IAS algorithm \eqref{eq:GSIAS algorithm} converge at least linearly to $\hat{u}$ in the Mahalanobis norm $\|\cdot \|_{D_{\hat{\theta}}}$. Furthermore, $\hat{u}$ is the global minimizer of $\mathsf{F}$ in \eqref{eq:objective over u} with 
        \begin{align}\label{eq: GS objective}
    \mathsf{R}_\eta(u):=\frac{1}{\sqrt{2}}\sum_{j=1}^k
    \biggl[ \frac{\|u_{g_j}\|^2_{C_j}}{2f_j(u)}  + f_j(u) - \eta \log f_j(u) \biggr], \qquad  f_j(u):=\frac{\eta}{2}+\sqrt{\frac{\eta^2}{4}+\frac{\|u_{g_j}\|_{C_j}^2}{2}}.
\end{align}
\end{theorem}
For completeness, we include a proof of Theorem \ref{thm:GSIAS computational properties} in Appendix \ref{appendix:computationalproperties}.
\end{example}

\begin{example}[Hierarchical Models for Sparse Representations and the O-IAS Algorithm]\label{ex:OIAS Example}
Our final motivating example is a novel hierarchical Bayesian model to promote the sparse representation of the unknown in some redundant dictionary, $W\in \R^{d\times k}$.
In this model, we assume the dictionary to be overcomplete, meaning that $k\gg d$. For simplicity, we assume that $W$ is a \textit{tight frame} \cite{waldron2018introduction} normalized so that $WW^T=I$. Relevant examples of such dictionaries include Gabor frames \cite{mallat1999wavelet}, curvelet frames \cite{candes2004new}, and wavelet frames \cite{ron1997affine}. To encode our a priori information that the coefficients of $u$ in $W$ are sparse, we place a Gaussian prior on $W^Tu$ and gamma hyperpriors on the variances,

\begin{align}\label{eq:hierarchicalmodelO-IAS}
    \begin{split}
    W^Tu| \theta &\sim \Nc \Bigl(0, \frac{\sqrt{2}}{\lambda n} D_\theta \Bigr), \qquad D_{\theta}=\text{diag}(\theta),\\
    \theta_j &\sim \text{Gamma}\left(\frac{\sqrt{2}\theta_j^*}{\lambda n}, \frac{3}{2}+\frac{\sqrt{2}}{2}\lambda n \eta\right), \quad 1\leq j\leq k,
    \end{split}
\end{align}
with scale parameters $\frac{\sqrt{2}\theta_j^*}{\lambda n}$ and shape parameters $\frac{3}{2}+\frac{\sqrt{2}}{2}\lambda n \eta$ for $\eta,\lambda>0$. Here, the auxiliary parameter $\theta\in \R^k$ controls the variance of the coefficients of $u$ in the dictionary $W$, and consequently the level of sparsity imposed upon the coefficients. From \eqref{eq:hierarchicalmodelO-IAS} we have the joint prior distribution 
\begin{align*}
    \pi(W^Tu,\theta)\propto \exp \left(-\frac{\lambda n}{2\sqrt{2}}\sum_{j=1}^k\frac{(W^T u)_j^2}{\theta_j}-\frac{\lambda n}{\sqrt{2}}\sum_{j=1}^k\left[\frac{\theta_j}{\theta_j^*}-\eta \log \frac{\theta_j}{\theta_j^*}\right] \right).
\end{align*}
For the convenience of our analysis in the sequel, we assume the scale parameters are chosen such that $\theta_j^*=\frac{n}{\|(AW)_j\|_2^2}$. Making $\left(\theta_j, (W^Tu)_j,(AW)_j \right) \mapsto \left(\frac{\theta_j}{\theta_j^*},\frac{(W^Tu)_j}{\sqrt{\theta_j^*}},(AW)_j\sqrt{\theta_j^*} \right),$ we see that our assumption on the scale parameters is equivalent to assuming that the forward map is normalized so that $\frac{\|(AW)_j\|_2}{\sqrt{n}}=1$ with $\theta_j^*=1$. Consequently, we consider a posterior with negative log-density given, up to an additive constant, by  
\begin{equation}\label{eq:O functional}
\mathsf{J}(u,\theta) := 
\mathrlap{  \underbrace{ \phantom{  \frac{1}{2n}\| y  - Au\|_2^2 + 
 \frac{\lambda}{2\sqrt{2}} \sum_{j=1}^k\frac{(W^Tu)_j^2}{\theta_j} }}_{(a)}}  \frac{1}{2n}\| y  - Au\|_2^2   
+  \overbrace{  \frac{\lambda}{2\sqrt{2}} \sum_{j=1}^k\frac{(W^Tu)_j^2}{\theta_j} + \frac{\lambda}{\sqrt{2}} \sum_{j=1}^k\left(\theta_j-\eta \log \theta_j \right) }^{(b)} \,\, .
\end{equation}
To compute the MAP estimator, we use the following overcomplete iterative alternating scheme (O-IAS), which minimizes  $(a)$ for the $u$-update and $(b)$ for the $\theta$-update:
\begin{align}\label{eq:OIAS algorithm}
\begin{split}
    u^{\ell+1}&:=\left(\frac{A^TA}{n}+\frac{\lambda}{\sqrt{2}}WD_{\theta^\ell}^{-1}W^T\right)\frac{A^Ty}{n},\\
    \theta_j^{\ell+1}&:=\frac{\eta}{2}+\sqrt{\frac{\eta^2}{4}+\frac{(W^Tu)_j^2}{2}}, \qquad 1 \le j \le k.
    \end{split}
\end{align}
The objective \eqref{eq:O functional} and the accompanying O-IAS algorithm \eqref{eq:OIAS algorithm} admit a number of key computationally convenient properties. 
\begin{theorem}\label{thm:OIAS computational properties}
    For $\eta>0$, the objective function $\J$ in \eqref{eq:O functional} is strictly convex over $\R^d\times \R^k_+$, and thus admits a unique MAP estimator $(\hat{u},\hat{\theta})$. Moreover, the $u$-iterates $u^\ell$ of the O-IAS algorithm converge at least linearly to $\hat{u}$ in the Mahalanobis norm $\|\cdot \|_{(WD_{\hat{\theta}}^{-1}W^T)^{-1}}.$ Furthermore, $\hat{u}$ is the global minimizer of $\mathsf{F}$ in \eqref{eq:objective over u} with 
    \begin{align}\label{eq:overcomplete objective}
        \mathsf{R}_{\eta}(u)=\frac{1}{\sqrt{2}}\sum_{j=1}^k\left[\frac{(W^Tu)_j}{2f_j(u)}+f_j(u)-\eta \log f_j(u)\right], \quad \quad f_j(u)=\frac{\eta}{2}+\sqrt{\frac{\eta^2}{4}+\frac{(W^Tu)_j^2}{2}}.
    \end{align}
\end{theorem}
For completeness, we include a proof of Theorem \ref{thm:OIAS computational properties} in Appendix \ref{appendix:computationalproperties}.
\end{example}

\section{Unified Theoretical Framework}\label{sec:GeneralTheory}
This section develops a framework to obtain reconstruction bounds for hierarchical Bayesian models under appropriate model assumptions. In Subsection \ref{ssec:approxdecomp} we introduce the notion of approximate decomposabilility, and show that this property is satisfied by the regularizers in the examples in Section \ref{sec:Problem Formulation}. Then, in Subsection \ref{ssec:maintheorem} we present the main result, Theorem \ref{thm:reconstruction error theorem}.

\subsection{Approximate Decomposability}\label{ssec:approxdecomp}
We introduce a \textit{model subspace} $\M\subset \R^d$ to capture the sparsity constraints of the problem. We refer to the orthogonal complement $\M^\perp$ of the model subspace as the \textit{perturbation subspace}. Following \cite{negahban2012unified,wainwright2019high}, we say that a regularizer $\mathsf{R}:\R^d\to \R^+$ is decomposable with respect to $\M$ if 
\begin{align}\label{eq:decomposability}
\mathsf{R}(u + v)=\mathsf{R}(u)+\mathsf{R}(v), \qquad \forall u \in \M, \forall v \in \M^{\perp}.
\end{align}
For norm-based regularizers, it necessarily holds that $\mathsf{R}(u+v)\le \mathsf{R}(u)+\mathsf{R}(v).$ Decomposability requires that equality holds for perturbations $v\in \M^\perp$ away from vectors $u\in \M,$ thus ensuring that deviations from the model subspace along the perturbation subspace are heavily penalized. The regularizers $\mathsf{R}_\eta$ discussed in Section \ref{sec:Problem Formulation} determined by the hierarchical priors are not decomposable over a natural model subspace. However, in the limit as $\eta\to 0$, each of these regularizers approach a decomposable regularizer. Motivated by this observation, we introduce the following notion of \textit{approximate decomposability}.
\begin{definition}\label{def:approximate decomposability}
    A family of regularizers $\mathsf{R}_{\eta}:\R^d\to \R^+$ parameterized by $\eta>0$ is approximately decomposable with respect to $\M$ if there exists a norm-based regularizer $\mathsf{R}$ decomposable with respect to $\M$ and  non-negative  functions $c_i^L(\eta)$ and $c_i^U(\eta),$ $i=1,2,$ such that, for any $u\in \R^d,$
    \begin{align}\label{eq: approximate decomposability bounds}
        \left(1-c_1^L(\eta) \right)\mathsf{R}(u)-c_2^L(\eta)\leq \mathsf{R}_\eta(u)\leq \left(1+c_1^U(\eta) \right)\mathsf{R}(u)+c_2^U(\eta),
    \end{align}
    where  $c_i^L(\eta)\to 0$ and $c_i^U(\eta)\to 0$ as $\eta\to 0$.
\end{definition}
The functions $c_i^L$ and $c_i^U$ quantify how fast the regularizer $\mathsf{R}_\eta(u)$ converges to the decomposable regularizer $\mathsf{R}(u)$ as $\eta$ is taken to be close to zero.  We remark that other generalizations of decomposability, such as \textit{weakly decomposable norms} \cite{van2014weakly} and \textit{atomic norms} \cite{chandrasekaran2012convex,candes2013simple}, have previously been studied. Our notion of approximate decomposability differs most significantly from these other generalizations in that it does not require the regularizer itself be a norm, merely that the regularizer be close to a norm for small $\eta$.  We now discuss the choice of model subspace and establish approximate decomposability for the three example models in Section \ref{sec:Problem Formulation}.

\begin{example}[Approximate Decomposability: Hierarchical Models for Sparsity]
    The natural modelling assumption for the IAS algorithm is that of sparsity. Let $S\subset\{1,\ldots , d\}$ be a subset of indices with cardinality $|S|=s.$ We define the model subspace
    \begin{align}\label{eq:sparse vector model subspace}
        \M(S):=\left\{u\in \R^d : u_j=0 \text{ for all } j\notin S \right\}.
    \end{align}
    Any vector in the model subspace $\M(S)$ has at most $s$ nonzero entries. In this setting, the perturbation subspace is the space of vectors with support contained entirely outside of $S$:
    \begin{align*}
        \M^{\perp}(S):=\left\{u\in \R^d : u_j=0 \text{ for all } j\in S \right\}.
    \end{align*}
    One can verify that, for any subset of indices $S$, the $\ell_1$ norm $\mathsf{R}(u)=\|u\|_1$ is decomposable with respect to $\M(S)$. The following lemma shows that the regularizer $\mathsf{R}_{\eta}$ defined in \eqref{eq:regularizer} is approximately decomposable with respect to $\M(S).$  \begin{lemma}\label{lemma:decomposabileybityIASregularizer}
Consider, for $0<\eta<\frac{1}{2},$ the family of regularizers $\mathsf{R}_\eta$ in \eqref{eq:regularizer}.
For any $u \in \R^d,$
\begin{equation}\label{eq:IASdecombounds}
    \Bigl(1 - \frac{\eta}{2}\Bigr) \| u \|_1 - \frac{d\eta}{\sqrt{2}} \le \mathsf{R}_\eta(u) \le \| u \|_1 + \frac{d\eta}{\sqrt{2}}(1 - \log\eta) .
\end{equation}
\end{lemma}     
    The proof can be found in Appendix \ref{appendix:exappdec}.
    In the notation of Definition \ref{def:approximate decomposability}, approximate decomposability holds with $c_1^L(\eta)=\frac{\eta}{2}$, $c_2^L(\eta)=\frac{d\eta}{\sqrt{2}}$, $c_1^U(\eta)=0,$ and $c_2^U(\eta)=\frac{d\eta}{\sqrt{2}}(1-\log\eta).$
\end{example}

\begin{example}[Approximate Decomposability: Hierarchical Models for Group Sparsity]
    We partition the index set into groups $\G=\{g_1,\ldots g_k\}$ where each $g_j\subset \{1,\ldots,d\}$, $g_i\cap g_j=\emptyset$ for $i\neq j$, and $\cup_{j=1}^kg_j=\{1,\ldots,d\}.$ We denote the size of each group as $|g_j|=p_j.$ For any set of group indices $S_\G\subset \G$, we define the model subspace
    \begin{align*}
    \M(S_\G):=\left\{u\in \R^d : u_g=0 \text{ for all }g\notin S_\G\right\}. 
    \end{align*}
    This is precisely the subspace of vectors supported only on groups in the index set $S_\G$. Conversely, the perturbation subspace is the space of vectors with support contained in groups not in $S_\G$:
    \begin{align*}
        \M^{\perp}(S_{\G}):=\left\{u\in \R^d : u_g=0 \text{ for all }g\in S_\G\right\}.
    \end{align*}
    A decomposable regularizer for this model subspace is the \textit{group sparse norm} \cite{yuan2006model}, given by
    \begin{align*}
        \mathsf{R}(u)=\sum_{j=1}^k\|u_{g_j}\|_{C_j},
    \end{align*}
    where $C_j\in \R^{p_j\times p_j}$ is a positive definite matrix. Indeed, we see that, for any $u\in \M(S_\G)$ and $v\in \M^{\perp}(S_\G),$ it holds that
    \begin{align*}
        \mathsf{R}(u+v)=\sum_{g_j\in \G}\|u_{g_j}\|_{C_j}+\sum_{g_j\notin \G}\|v_{g_j}\|_{C_j}=\mathsf{R}(u)+\mathsf{R}(v).
    \end{align*}
    The following lemma shows that the regularizer $\mathsf{R}_{\eta}$ defined in \eqref{eq: GS objective} is approximately decomposable with respect to $\M(S_\G)$. 
    \begin{lemma}\label{lemma:decomposabilityGSIASregularizer}
    Consider, for $0<\eta<\frac{1}{2},$ the family of regularizers $\mathsf{R}_\eta$ in \eqref{eq: GS objective}.
For any $u \in \R^d,$
\begin{equation*}
    \Bigl(1 - \frac{\eta}{2}\Bigr)\sum_{j=1}^k\|u_{g_j}\|_{C_j} - \frac{k\eta}{\sqrt{2}} \le \mathsf{R}_\eta(u) \le \sum_{j=1}^k\|u_{g_j}\|_{C_j} + \frac{k\eta}{\sqrt{2}}(1 - \log\eta) .
\end{equation*}
\end{lemma} 
The proof can be found in Appendix \ref{appendix:exappdec}.
  In the notation of Definition \ref{def:approximate decomposability}, approximate decomposability holds with $c_1^L(\eta)=\frac{\eta}{2}$, $c_2^L(\eta)=\frac{k\eta}{\sqrt{2}}$, $c_1^U(\eta)=0$, and $c_2^U(\eta)=\frac{k\eta}{\sqrt{2}}(1-\log\eta).$
\end{example}

\begin{example}[Approximate Decomposability: Hierarchical Models for Sparse Representations] We consider the situation where the data generating parameter itself is not sparse, but admits a sparse representation in an overcomplete dictionary, $W\in \R^{d\times k},$ where $k\gg d$. As is common in the compressed sensing literature \cite{candes2011compressed}, we assume that $W$ is a tight frame such that $WW^T=I.$ We let $S_W\subset\{1,\ldots,k\}$ be a subset of indices of cardinality $|S_W|=s$. We then define the model subspace
\begin{align}\label{eq:sparse representation model subspace}
    \M(S_W):=\left\{u \in \R^d :(W^Tu)_j=0 \text{ for all } j\notin S_W \right\}.
\end{align}
    This is precisely the subspace of vectors whose representation in $W$ has support contained in $S_W$. The perturbation subspace is then given by 
    \begin{align*}
        \M^\perp(S_W):=\left\{u \in \R^d :(W^Tu)_j=0 \text{ for all } j\in S_W \right\}.
    \end{align*}
    A decomposable regularizer for this model subspace is given by
    \begin{align*}
        \mathsf{R}(u)=\|W^Tu\|_1,
    \end{align*}
    since, for any $u\in \M(S_W)$ and $v\in \M^\perp(S_W),$ it holds that
    \begin{align*}
        \mathsf{R}(u+ v)=\|W^T u +W^Tv\|_1=\|W^T u\|_1+\|W^Tv\|_1=\mathsf{R}(u)+\mathsf{R}(v).
    \end{align*}
The following lemma shows that the regularizer $\mathsf{R}_\eta$ defined in \eqref{eq:overcomplete objective} is approximately decomposable with respect to $\M(S_W).$ 
\begin{lemma}\label{lemma:decomposabilityOIASregularizer}
    Consider, for $0<\eta<\frac{1}{2},$ the family of regularizers $\mathsf{R}_\eta$ in \eqref{eq:overcomplete objective}. 
For any $u\in \R^d$,
\begin{equation*}
    \Bigl(1 - \frac{\eta}{2}\Bigr) \| W^T u \|_1 - \frac{k\eta}{\sqrt{2}} \le \mathsf{R}_\eta(u) \le \| W^Tu \|_1 + \frac{k\eta}{\sqrt{2}}(1 - \log\eta).
\end{equation*}
\end{lemma}
The proof can be found in Appendix \ref{appendix:exappdec}.
In the notation of Definition \ref{def:approximate decomposability}, approximate decomposability holds with $c_1^L(\eta)=\frac{\eta}{2}$, $c_2^L(\eta)=\frac{k\eta}{\sqrt{2}},$ $c_1^U(\eta)=0$, and $c_2^U(\eta)=\frac{k\eta}{\sqrt{2}}(1-\log\eta).$
\end{example}

We now show how the notion of approximate decomposability can be leveraged to control the error vector $\hat{\Delta}=\hat{u}-u^\star$. To do so, we recall that the dual norm of $\mathsf{R}$ with respect to the Euclidean inner product is defined by 
\begin{equation*}
    \mathsf{R}^*(u) := \sup_{v \in \R^d \setminus \{0 \}} \frac{ \langle u, v \rangle}{\mathsf{R}(v)} = \sup_{\mathsf{R}(v) \le 1} \langle u, v \rangle.
\end{equation*} In the following lemma, we show that if the regularization parameter $\lambda$ is taken to be large enough relative to the dual norm of the noise vector, then the reconstruction error lies in a specific set. For vector $u\in \R^d$ and subspace $\M\subset \R^d,$ we denote by $u_\M$ the Euclidean projection of $u$ onto $\M.$

\begin{lemma}\label{lemma:decomposabilityproperty}
    Let $\hat{u}$ be a minimizer of the objective
    \begin{align}\label{eq:objective}
        \mathsf{F}(u)=\frac{1}{2n}\|y-Au\|_2^2+\lambda \mathsf{R}_{\eta}(u),
    \end{align}
    where $\mathsf{R}_\eta$ is an approximately decomposable regularizer with respect to $\M$. If
        $\lambda \geq \frac{2}{n}\mathsf{R}^*\left( A^T\eps\right)$
    and $\eta$ is small enough so that $c_1^L(\eta)<\frac{1}{4}$, then the error $\hat{\Delta}= \hat{u}-u^\star$ lies in the set
    \begin{align*}
        \mathbb{C}_{\eta}(\M):=\left\{\Delta\in \R^d: \mathsf{R}(\Delta_{\M^\perp})\leq 7\mathsf{R}(\Delta_\M)+8\mathsf{R}(u^\star_{\M^\perp})+4c_1(\eta)\mathsf{R}(u^\star_\M)+4c_2(\eta) \right\},
    \end{align*}
    where $c_i(\eta) :=c_i^L(\eta) + c_i^U(\eta).$
\end{lemma}  
The proof can be found in Appendix \ref{appendix:proofgeneralbounds}.
The fact that the error lies in the set $\mathbb{C}_{\eta}(\M)$ allows us to control the components of the error outside the model subspace by the components of the error within the model subspace along with terms that get smaller as the true vector is assumed to lie within the model subspace and as $\eta$ is taken to be smaller. In the next subsection we leverage this lemma to get a bound on the reconstruction error.

\subsection{General Bound on the Reconstruction Error}\label{ssec:maintheorem}
To bound the reconstruction error, we will need to quantify the cost of translating between bounds on a vector in $\M$ given by $\|\cdot \|_2$ and bounds given by $\mathsf{R}(\cdot).$ To that end, we will rely on the notion of  \textit{subspace Lipschitz constant}.
\begin{definition}
    For a subspace $\M\in \R^d$, the subspace Lipschitz constant with respect to $(\mathsf{R},\|\cdot\|_2)$ is given by 
    \begin{align*}
        \Psi(\M):=\sup_{u\in \M \backslash \{0\}}\frac{\mathsf{R}(u)}{\|u\|_2}.
    \end{align*}
\end{definition}
We will also need a condition on the forward map. In the high-dimensional setting, the likelihood term $\|y-Au\|_2^2$, necessarily has $d-n$ directions in which it has no curvature up to second order. Consequently, we need an assumption to guarantee that in the relevant directions, chiefly the directions given by vectors in $\C_{\eta}(\M)$, the likelihood has significant curvature at $u^\star.$ One such assumption is that of \textit{restricted strong convexity}.

\begin{definition}
    For a given regularizer $\mathsf{R}$, a linear map $A:\R^d\to \R^n$ satisfies the restricted strong convexity (RSC) condition with curvature $\kappa$ and tolerance $\tau^2$ if 
    \begin{align}\label{eq:RSC}
        \frac{\|A\Delta\|_2^2}{2n}\geq \frac{\kappa}{2}\|\Delta\|_2^2-\tau^2\mathsf{R}^2(\Delta), \qquad \forall \Delta\in \R^d.
    \end{align}
\end{definition}
If the matrix $A$ were not rank degenerate, then the RSC condition would hold with $\tau=0,$ which recovers the usual definition of strong convexity. In the high-dimensional setting, this cannot be the case, and a nonzero tolerance term is required. In general, we expect the tolerance to be decreasing in the number $n$ of observations.\\

 Given these definitions, we now state our main result, which generalizes the bounds in \cite[Chapter 9]{wainwright2019high} to hold for approximately decomposable regularizers. The proof can be found in Appendix \ref{appendix:proofgeneralbounds}.
\begin{theorem}\label{thm:reconstruction error theorem}
    Let $u^\star$ be the true parameter generating the data via \eqref{eq:frequentist inverse problem} and let $\hat{u}$ be the minimizer of the objective \eqref{eq:objective}. Assume that $\mathsf{R}_\eta$ is approximately decomposable with respect to $\M$ and that $\eta$ is small enough such that $c_1^L(\eta)<\frac{1}{4}$ 
    and that the forward map $A$ satisfies the RSC condition with curvature $\kappa$ and tolerance $\tau^2$. Then, if $\tau\Psi(\M)\leq \frac{\sqrt{\kappa}}{32}$ and $\lambda\geq 2\mathsf{R}^*(\frac{1}{n}A^T\eps),$ we have the bound
        \begin{align}\label{eq:generalbound}
        \|\hat{u}-u^\star\|_2^2\lesssim \frac{\lambda^2}{\kappa^2}\Psi^2(\M)+ \frac{\tau^2}{\kappa}\mathsf{R}(u^\star_{\M^{\perp}})^2+ \frac{\lambda}{\kappa} \mathsf{R}(u^\star_{\M^\perp})+\mathcal{E}_\eta,
    \end{align}
    where
    \begin{equation*}
        \mathcal{E}_{\eta}:= \frac{\tau^2}{\kappa}\Bigl( c_1(\eta)^2\mathsf{R}(u^\star_\M)^2+ c_2(\eta)^2\Bigr)+ \frac{\lambda}{\kappa}\Bigl( c_1(\eta)\mathsf{R}(u^\star_\M)+c_2(\eta)\Bigr),
    \end{equation*}
  and $c_i(\eta):=c_i^L(\eta) + c_i^U(\eta), \, i = 1, 2.$   
\end{theorem}
We emphasize that this bound is deterministic, and makes no specific distributional assumptions on the noise or forward map. The first three terms are precisely the terms that appear in the bound proven in \cite{wainwright2019high} for exactly decomposable regularizers. The first term, 
\begin{align*}
    \mathcal{E}_{\text{est}}:=\frac{\lambda^2}{\kappa^2}\Psi^2(\M),
\end{align*} can be thought of as the statistical estimation error which captures the inherent difficulty of estimating a parameter in $\M$. The second and third terms together, 
\begin{align*}
    \mathcal{E}_{\text{app}}:=\frac{\tau^2}{\kappa}\mathsf{R}(u^\star_{\M^{\perp}})^2+ \frac{\lambda}{\kappa} \mathsf{R}(u^\star_{\M^\perp}),
\end{align*}  represent the approximation error. These terms capture the modelling error incurred when the data-generating parameter does not lie exactly in $\M$, thus quantifying the amount of model misspecification. Notice that if $u^\star\in \M$, then these terms vanish. The final term,
\begin{equation*}
        \mathcal{E}_{\eta}:= \frac{\tau^2}{\kappa}\Bigl( c_1(\eta)^2\mathsf{R}(u^\star_\M)^2+ c_2(\eta)^2\Bigr)+ \frac{\lambda}{\kappa}\Bigl( c_1(\eta)\mathsf{R}(u^\star_\M)+c_2(\eta)\Bigr),
    \end{equation*}
is novel to our bound, and captures the error incurred by the approximate decomposability of the regularizer. This term can be made arbitrarily small with $\eta$, which is user-specified.

\section{Reconstruction Error for Sparsity-Promoting Models}\label{sec:convergence}
In this section we apply the unified theoretical framework set forth in Theorem \ref{thm:reconstruction error theorem} to each of the examples of hierarchical models discussed in Section \ref{sec:Problem Formulation}. The results in this section provide the first known reconstruction error bounds for these models. Proofs will be deferred to Section \ref{sec:proofs}. 
\subsection{Hierarchical Models for Sparsity}\label{ssec:IAS}
Here we apply our general theory to the sparsity-promoting hierarchical model presented in Example \ref{ex:IAS Example} under both a hard-sparsity assumption and a weaker $\ell^q$-sparsity assumption. For both of these types of sparsity, we first prove a deterministic result guaranteeing that under the RSC condition, the MAP estimator will be close to the true data generating parameter as long as we regularize enough relative to the noise level. We emphasize that these results do not make specific assumptions on the noise distribution or the forward map, besides the RSC condition. Applying the deterministic results to the probabilistic setting of Gaussian noise and Gaussian forward maps, we derive sharp rates in terms of only the dimension, sample size, and the parameter $\eta$. These results hold with high probability.\\

The first assumption we consider is that of hard sparsity. We assume that the true parameter generating the data has support of cardinality $s$, where $s\ll d$. The following result holds for any $s$-sparse $u^\star$ so long as the forward map satisfies the RSC condition and the regularization parameter is taken to be large enough.
\begin{theorem}\label{thm:deterministic error bound IAS}
    Let $u^\star$ be the true parameter generating data via \eqref{eq:inverse problem} and suppose that $\text{supp}(u^\star)\subset S$ where $|S|=s.$ Let $(\hat{u},\hat{\theta})$ be the unique minimizer of the objective function \eqref{eq:functional} with $0<\eta< 1/2$. Assume that $\lambda \geq \frac{2\|A^T\eps\|_{\infty}}{n}$ and that $A$ satisfies the RSC condition \eqref{eq:RSC} with curvature $\kappa$ and tolerance  $\tau^2$. Then, if $\tau \sqrt{s}\leq \frac{\sqrt{\kappa}}{32}$ it holds that
    \begin{align}\label{eq:reconstruction error bound}
        \|\hat{u}-u^\star\|_2^2\lesssim s\frac{\lambda^2}{\kappa^2}+ \mathcal{E}_\eta,
    \end{align}
where
\begin{equation*}
    \mathcal{E}_\eta := \frac{\tau^2}{\kappa}\left(\eta^2\|u^\star_{S}\|_1^2+\frac{d^2\eta^2}{2}(2-\log\eta)^2 \right)+ \frac{\lambda}{\kappa} \left(\eta\|u^\star_S\|_1+\frac{d\eta}{\sqrt{2}}(2-\log\eta) \right).
\end{equation*}
\end{theorem}
We emphasize that this theorem is deterministic. It makes no specific distributional assumptions on the noise or forward map, only requiring that the regularization parameter is large enough relative to the noise and that the forward map satisfies the RSC condition. The bound illustrates that, as long as the condition $\lambda \geq \frac{2\|A^T\eps\|_{\infty}}{n}$ is satisfied, we would like to take the regularization parameter $\lambda,$ the RSC tolerance $\tau$, and the parameter $\eta$ to be as small as possible. We remark that $\eta$ is user supplied and can be made arbitrarily small, but doing so may slow down the convergence of the IAS algorithm. The other terms in the bound, $\lambda$ and $\tau$, inherently depend on the noise and the forward map. As an illustrative example, we explicitly characterize $\lambda$ and $\tau$ under the assumptions of Gaussian noise and a Gaussian forward map in the following corollary:
\begin{corollary}\label{thm:Gaussian Corollary}
    Let the forward map $A\in \R^{n\times d}$ have rows distributed i.i.d. according to $\Nc(0,\Sigma)$, and assume the columns of $A$ are normalized such that $\frac{\|A_j\|_2}{\sqrt{n}}= 1$ for all $j=1,\ldots, d.$ Assume the noise is given by $\eps \sim \Nc(0, I)$ and that $(\hat{u},\hat{\theta})$ is the unique minimizer of the IAS objective functional \eqref{eq:functional} with $\lambda=4\sqrt{\frac{\log d}{n}}$ and $0<\eta< 1/2$. If $u^\star$ is $s$-sparse, then there exists a constant $c_{\Sigma}$ depending only on $\Sigma$, and universal constants $c$ and $c'$ such that if $n>c_{\Sigma}s\log d$ it holds with probability at least $1-c\exp(-c'n\lambda^2)$ that
    \begin{align*}
        \|\hat{u}-u^\star\|_2^2\lesssim c_{\Sigma}\frac{s\log d}{n}+\min \left\{ \phi(\eta),\phi(\eta)^2\right\},
    \end{align*}
    where
    \begin{align*}
        \phi(\eta):=\sqrt{\frac{\log d}{n}}\left( \eta\|u^\star_S\|_1+\frac{d\eta}{\sqrt{2}}(2-\log\eta)\right).
    \end{align*}
\end{corollary}
Notice that the term $\phi(\eta)$ can be made arbitrarily small by choosing a small enough $\eta$. Taking $\eta$ small enough to match the first term in the bound, the MAP estimator for the hierarchical model in Example \ref{ex:IAS Example} satisfies the non-asymptotic bound
\begin{align*}
    \|\hat{u}-u^\star\|_2\lesssim \sqrt{c_{\Sigma}\frac{s\log d}{n}},
\end{align*}
which agrees with known non-asymptotic bounds for the Lasso estimator \cite{bickel2009simultaneous,meinshausen2009lasso}. The assumptions of Gaussianity in this corollary are largely for theoretical convenience, and are not strictly necessary. Verifying the RSC condition for the forwards maps of interest in the inverse problems community remains a direction of further research.\\

Instead of assuming the true parameter is exactly $s$-sparse, one can make the more general assumption of $\ell^q$ sparsity. That is, we assume
\begin{align*}
    u^\star\in \mathbb{B}_q(R_q):=\left\{u\in \R^d :\sum_{i=1}^d|u_i|^q\leq R_q\right\},
\end{align*}
where $R_q>0$ controls the level of sparsity. Notice that if $q=0$ and $R_q=s$, we recover the hard-sparsity assumption. For $q>0$ this assumption enforces a decay rate on the entries of $u^\star.$ Under $\ell^q$ sparsity, we choose our model subspace to be vectors with support $S_{\delta}$, where
\begin{align}\label{eq:S delta}
    S_{\delta}:=\left\{j\in \{1,\ldots , d\} :|u_j^\star|>\delta \right\}.
\end{align}
This model subspace is well defined for any positive $\delta$. 
We will see that our bound on the reconstruction error depends on $\delta$. Note that we define $\delta$ only as a theoretical tool, since the model subspace is not involved methodologically when computing the MAP estimator, $\hat{u}$. We emphasize that in the setting of $\ell^q$ sparsity, we have that $u^\star_{S^C_\delta}\neq 0$ in general, and consequently the model error terms in the general bounds are non-negligible. We now present a deterministic result analogous to Theorem \ref{thm:deterministic error bound IAS} under the $\ell^q$-sparsity assumption.
\begin{theorem}\label{thm:deterministic bound IAS LQ}
    Let $u^\star$ be the true parameter generating data via \eqref{eq:inverse problem} and suppose that $u^\star\in \mathbb{B}_q(R_q).$ Let $(\hat{u},\hat{\theta})$ be the unique minimizer of the objective function \eqref{eq:functional} with $0<\eta< 1/2$. Assume that $\lambda \geq \frac{2\|A^T\eps\|_{\infty}}{n}$ and that $A$ satisfies the RSC condition \eqref{eq:RSC} with curvature  $\kappa$ and tolerance $\tau^2$. Then, for any $\delta>0$, if $\tau \sqrt{R_q \delta^{-q}}\leq \frac{\sqrt{\kappa}}{32}$ it holds that 
    \begin{align*}
        \|\hat{u}-u^\star\|_2^2\lesssim \frac{\lambda^2}{\kappa^2}R_q\delta^{-q}+ \frac{\tau^2}{\kappa}R_q\delta^{1-q}+ \frac{\lambda}{\kappa} R_q \delta^{1-q} + \mathcal{E}_\eta,
    \end{align*}
    where
        \begin{equation*}
       \mathcal{E}_\eta :=  \frac{\tau^2}{\kappa} \Bigl(\eta^2R_q^{2/q}+ d^2\eta^2(2-\log \eta)^2\Bigr) +\frac{\lambda}{\kappa}\Bigl(\eta R_q^{1/q}+d\eta(2-\log \eta)\Bigr). 
    \end{equation*}
\end{theorem}
Since this deterministic error bound holds for any positive $\delta,$ Theorem \ref{thm:deterministic bound IAS LQ} provides a family of bounds parameterized by $\delta.$ Under noise and forward map assumptions that yield particular values for $\lambda$, $\kappa$, and $\tau$, one can optimize $\delta$ to yield the sharpest bound. 
We do so in the following corollary under the assumptions of Gaussian noise and a Gaussian forward map.
\begin{corollary}\label{thm:Gaussian Corollary IAS LQ}
    Let the forward map $A\in \R^{n\times d}$ have rows distributed i.i.d. according to $\Nc(0,\Sigma)$, and assume the columns of $A$ are normalized such that $\frac{\|A_j\|_2}{\sqrt{n}}= 1$ for all $j=1,\ldots, d.$ Assume the noise is given by $\eps \sim \Nc(0, I)$ and that $(\hat{u},\hat{\theta})$ is the unique minimizer of the IAS objective functional \eqref{eq:functional} with $\lambda=4\sqrt{\frac{\log d}{n}}$ and $0<\eta< 1/2$. If $u^\star\in \mathbb{B}_q(R_q)$, then there exists a constant $c_{\Sigma}$ depending only on $\Sigma$, and universal constants $c$ and $c'$ such that if  $n>c_{\Sigma}\log d R_q^{1-q/2}$  it holds with probability at least $1-c\exp(-c'n\lambda^2)$ that
\begin{align*}
    \|\hat{u}-u^\star\|_2^2\lesssim  c_{\Sigma}R_q\left(\frac{\log d}{n} \right)^{1-\frac{q}{2}}+c_{\Sigma}\min \left\{\phi(\eta),\phi(\eta)^2\right\},
\end{align*}
where
\begin{align*}
    \phi(\eta):=\sqrt{ \frac{\log d}{n}}\left(\eta R_q^{\frac{1}{q}}+\frac{ d \eta}{\sqrt{2}}(2-\log \eta)\right).
\end{align*}
\end{corollary}
This corollary suggests that $\eta$ should be taken small enough so that $\phi(\eta)$ matches the rate of $R_q\left( \frac{\log d}{n}\right)^{1-\frac{q}{2}}$. In particular, for any $\eta$ such that
\begin{align*}
    \eta \log \frac{1}{\eta} \leq \min \left\{\frac{R_q}{d}\left(\frac{\log d}{n} \right)^{\frac{1-q}{2}},R_q^{1-\frac{1}{q}}\left(\frac{\log d}{n} \right)^{\frac{1-q}{2}}
\right\},
\end{align*}
the MAP estimator of the hierarchical Bayesian model in Example \ref{ex:IAS Example} achieves up to constant factors the minimax rate for estimation under $\ell^q$ sparsity \cite{raskutti2011minimax}, given by
\begin{align*}
    \|\hat{u}-u^\star\|_2^2\lesssim  R_q\left(\frac{\log d}{n} \right)^{1-\frac{q}{2}}.
\end{align*} 
That is, with this choice of $\eta$, the MAP estimator achieves the best possible reconstruction error.
\subsection{Hierarchical Models for Group Sparsity}\label{ssec:GSIAS}
In this subsection we derive non-asymptotic bounds on the reconstruction error for the group-sparsity-promoting hierarchical model considered in Example \ref{ex:GSIAS Example}. We prove reconstruction bounds under a weak $\ell^q$-group-sparsity assumption. We suppose that
\begin{align*}
    u^\star\in \mathbb{B}_q(R_q,\G):=\left\{u\in \R^d : \sum_{j=1}^k\|u_{g_j}\|_{C_j}^q\leq R_q \right\}.
\end{align*}
We emphasize that in the case where $q=0$ and $R_q=s$, we recover the hard-group-sparsity assumption that $u_g=0$ for all $g\in S_\G$ with $|S_\G|=s.$
We consider the set of group indices with norm larger than $\delta>0$, defined by
\begin{align}\label{eq: S G delta}
    S_{\G,\delta}:=\left\{g_j\in \G : \|u_{g_j}^\star\|_{C_j}> \delta\right\}.
\end{align}
We now present a deterministic result that holds for any positive $\delta$ so long as the regularization parameter is taken to be large enough relative to the noise and the forward map satisfies an RSC condition.
\begin{theorem}\label{thm: deterministic error bound GSIAS}
    Let $u^\star$ be the true parameter generating the data via \eqref{eq:frequentist inverse problem} and suppose that $u^\star\in \mathbb{B}_q(R_q,\G)$. Let $(\hat{u},\hat{\theta})$ be the minimizer of the objective function \eqref{eq:GS functional} with $0< \eta < 1/2.$ Assume that $\lambda \geq \max_{g_j\in \G}\frac{2}{n}\|(A^T\eps)_{g_j}\|_{C_j^{-1}},$
    and that $A$ satisfies the RSC condition \eqref{eq:RSC} with curvature $\kappa$ and tolerance $\tau^2$. Then, for any $\delta>0$, if $\tau \sqrt{R_q\delta^{-q}}\leq \frac{\sqrt{\kappa}}{32}$ it holds that
    \begin{align*}
        \|\hat{u}-u^\star\|_2^2 \lesssim \frac{\lambda^2}{\kappa^2}R_q\delta^{-q}+\frac{\tau^2}{\kappa}R_q\delta^{1-q}+\frac{\lambda}{\kappa}R_q\delta^{1-q}+\mathcal{E}_\eta,
    \end{align*}
    where
    \begin{align*}
        \mathcal{E}_\eta :=\frac{\tau^2}{\kappa}\left(\eta^2R_q^{\frac{2}{q}}+k^2\eta^2(2-\log \eta)^2 \right) +\frac{\lambda}{\kappa}\left( \eta R_q^{\frac{1}{q}}+k\eta(2-\log \eta)\right).
    \end{align*}
\end{theorem}
We remark that if $q=0$ and $R_q=s$, this theorem recovers bounds for the hard-sparsity assumption that $u^\star\in \M(S_{\G})$:
    \begin{align*}
        \|\hat{u}-u^\star\|_2^2\lesssim s\frac{\lambda^2}{\kappa^2}+\mathcal{E}_{\eta},
    \end{align*}
    where
    \begin{align*}
        \mathcal{E}_\eta :=\frac{\tau^2}{\kappa}\left(\eta^2\left(\sum_{j\in S_\G}\|u^\star_{g_j}\|_{C_j}\right)^2+k^2\eta^2(2-\log \eta)^2 \right)+\frac{\lambda}{\kappa}\left(\eta\sum_{j\in S_\G}\|u^\star_{g_j}\|_{C_j}+k\eta(2-\log\eta) \right).
    \end{align*}
    As was the case for the models discussed in the previous section, under distributional assumptions on the noise and forward map, we can use this deterministic bound to sharply characterize non-asymptotic bounds on the reconstruction error with high probability. For simplicity of presentation, we assume that $C_j=I$ for $1\leq j\leq k$. This assumption is straightforward to remove, but doing so is notationally cumbersome. We require a block-normalization condition on $A$ that generalizes the column-normalization assumption used in the previous subsection. Given a group $g_j\in \G$ of size $p_j$, we consider the submatrix $A_{g_j}:\R^{n\times p_j}$, and require that
    \begin{align} \label{eq:block normalization}
        \frac{\|A_{g_j}\|_{op}}{\sqrt{n}}= 1, \quad j\in \{1,\ldots,k\}.
    \end{align}
    As remarked in Example \ref{ex:GSIAS Example}, this condition can be guaranteed by appropriately choosing the model hyperparameters. For the following corollary, we introduce the notation $p_{\max} :=\max_{j\in \{1,\ldots,k\}} p_j. $
    \begin{corollary}\label{thm: Gaussian Corollary GS}
        Let the forward map $A\in \mathbb{R}^{n\times d}$ have rows distributed i.i.d. according to $\Nc(0,\Sigma)$ and assume that the block normalization condition \eqref{eq:block normalization} holds. Assume that the noise is given by $\eps \sim \Nc(0,I)$ and that $(\hat{u},\hat{\theta})$ is the unique minimizer of the GS-IAS objective functional $\eqref{eq:GS functional}$ with $\lambda=\left(\sqrt{\frac{p_{\max}^{}}{n}
        }+\sqrt{\frac{\log k}{n}} \right)$ and $0<\eta<1/2.$ If $u^\star\in \mathbb{B}_q(R_q,\G),$ then there exists a constant $c_\Sigma$ depending only on $\Sigma$ and such that if $\sqrt{n}>c_{\Sigma}R_q^{1-q/2}\left(\sqrt{p_{\max}^{} }+\sqrt{\log k}\right) $ it holds with probability at least $1-2/k^2$ that
        \begin{align*}
            \|\hat{u}-u^\star\|_2^2\lesssim c_{\Sigma}R_q\left( \frac{p_{\max}}{n}+\frac{\log k}{n}\right)^{1-\frac{q}{2}} +\min \left\{\varphi(\eta),\varphi(\eta)^2\right\},
        \end{align*}
        where
        \begin{align*}
            \varphi(\eta):=\sqrt{\frac{p_{\max}+\log k}{n}}\left(\eta R_q^{\frac{1}{q}}+k\eta (2-\log \eta) \right).
        \end{align*}
    \end{corollary}
    If $q=0$ and $R_q=s$, we recover a bound under a hard-group-sparsity assumption:
    \begin{align*}
        \|\hat{u}-u^\star\|_2^2\lesssim \frac{s p_{\max}}{n}+\frac{s\log k}{n}+\max \left\{\varphi(\eta),\varphi(\eta)^2\right\},
    \end{align*}
    where 
    \begin{align*}
        \varphi(\eta):=\sqrt{\frac{p_{\max}+\log k}{n}}\left(\eta\sum_{j=1}^k\|u^\star_{g_j}\|_{C_j}+k\eta (2-\log \eta) \right).
    \end{align*}
    In both the $\ell^q$-group-sparsity and hard-group-sparsity settings, by taking $\eta$ to be small enough to match the estimation rate in the first half of the bound, we see that the hierarchical model in Example \ref{ex:GSIAS Example} and the GS-IAS algorithm can achieve the known bounds for reconstructing group sparse vectors \cite{negahban2012unified,lounici2009taking,huang2010benefit}.
\subsection{Hierarchical Models for Sparse Representations
}\label{ssec:OIAS}
Finally, we apply our theory to derive non-asymptotic reconstruction error bounds for the hierarchical model promoting sparse representations described in Example \ref{ex:OIAS Example}. We assume that the data-generating unknown $u^\star$ has a representation in a tight from $W\in \R^{d\times k}$, where $k$ may be much larger than $d$. We assume this representation lies in an $\ell^q$ ball:
\begin{align*}
    u^\star\in \mathbb{B}_{q}(R_q,W):=\left\{u\in \R^d : \sum_{j=1}^k|(W^Tu)_j|^q\leq R_q \right\}.
\end{align*}
As in the previous examples, we recover the hard-sparsity assumption that the support of $W^Tu^\star$ has cardinality $s$ in the case where $q=0$ and $R_q=s$. We define our model subspace as in \eqref{eq:sparse representation model subspace}, taking the index set to be
\begin{align}\label{eq: model subspace S W delta}
    S_{W,\delta}:=\left\{j\in \{1,\ldots ,k\} : (W^Tu^\star)_j>\delta\right\}.
\end{align}
We show a deterministic bound on the reconstruction error that holds for any positive $\delta$ provided that the regularization parameter is large enough relative to the noise and that the forward map satisfies the RSC condition.
\begin{theorem}\label{thm: deterministic error bound OSIAS}
    Let $u^\star$ be the true parameter generating the data via \eqref{eq:frequentist inverse problem} and suppose that $u^\star\in \mathbb{B}_q(R_q,W).$ Let $(\hat{u},\hat{\theta})$ be the minimizer of the objective function \eqref{eq:O functional} with $0<\eta<1/2.$ Assume that $\lambda \geq \frac{2\|W^TA^T\eps\|_{\infty}}{n}$ and that $A$ satisfies the RSC condition $\eqref{eq:RSC}$ with curvature $\kappa$ and tolerance $\tau^2$. Then, for any $\delta>0$, if $\tau \sqrt{R_q\delta^{-q}}\leq \frac{\kappa}{32}$ it holds that
    \begin{align*}
        \|\hat{u}-u^\star\|_2^2\lesssim \frac{\lambda^2}{\kappa^2}R_q\delta^{-q}+\frac{\tau^2}{\kappa}R_q\delta^{1-q}+\frac{\lambda}{\kappa}R_q\delta^{1-q}+\mathcal{E}_{\eta},
    \end{align*}
    where
    \begin{align*}
        \mathcal{E}_{\eta}:=\frac{\tau^2}{\kappa}\left(\eta^2R_q^{2/q}+k^2\eta^2(2-\log(\eta)^2 \right)+\frac{\lambda}{\kappa}\left( \eta R_q^{1/q}+k\eta(2-\log \eta)\right).
    \end{align*}
\end{theorem}
When $q=0$ and $R_q=s$ this result recovers a bound under the hard-sparsity assumption that $\text{supp}(W^Tu^\star)=S$ with $|S|=s:$
\begin{align*}
    \|\hat{u}-u^\star\|_2^2\lesssim s\frac{\lambda^2}{\kappa^2}+\mathcal{E}_{\eta},
\end{align*}
where 
\begin{align*}
    \mathcal{E}_{\eta}:=\frac{\tau^2}{\kappa}\left(\eta^2\|(W^Tu^\star)_S\|_1^2+k^2\eta^2(2-\log \eta)^2 \right)+\frac{\lambda}{\kappa}\left( \eta\|(W^Tu^\star)_S\|_1+k\eta(2-\log \eta)\right).
\end{align*}
Once again, by making assumptions on the distribution of the noise and forward map, we can derive sharp non-asymptotic bounds on the reconstruction error with high probability. 
\begin{corollary}\label{thm:Gaussian Corollary SR}
    Let the forward map $A\in \R^{n\times d}$ have rows distributed i.i.d. according to $\Nc(0,\Sigma)$ and assume that $A$ is normalized such that $\frac{\|(AW)_j\|_2}{\sqrt{n}}= 1$ for all $j=1,\ldots , k$. Assume the noise is given by $\eps \sim \Nc(0,I)$ and that $(\hat{u},\hat{\theta})$ is the minimizer of the O-IAS objective functional \eqref{eq:O functional} with $\lambda=4\sqrt{\frac{\log k}{n}}$ and $0<\eta<1/2.$ If $u^\star\in \mathbb{B}_Q(R_q,Q)$, then there exists a constant $c_{\Sigma}$ depending only on $\Sigma$ and universal constants $c$ and $c'$ such that if $n>c_{\Sigma} R_q^{1-q/2}\log k$ it holds with probability at least $1-c\exp(c'n\lambda^2)$ that
    \begin{align*}
        \|\hat{u}-u^\star\|_2^2\lesssim R_q\left( \frac{\log k}{n}\right)^{1-\frac{q}{2}}+\min \left\{\varphi(\eta),\varphi(\eta)^2 \right\},
    \end{align*}
    where
    \begin{align*}
        \varphi(\eta):=\sqrt{\frac{\log k}{n}}\left(\eta R_q^{1/q} +k\eta(2-\log \eta)\right).
    \end{align*}
\end{corollary}
The corresponding hard-sparsity bound under the assumption that $q=0$ and $R_q=s$ is 
\begin{align*}
    \|\hat{u}-u^\star\|_2^2\lesssim\frac{s\log k}{n}+\max\left\{\varphi(\eta),\varphi(\eta)^2\right\},
\end{align*}
where
\begin{align*}
    \varphi(\eta):=\sqrt{\frac{\log k}{n}}\left(\eta \|W^Tu^\star\|_1+k\eta(2-\log \eta) \right).
\end{align*}

\section{Proof of Reconstruction Error for Sparsity-Promoting Models}\label{sec:proofs}
This section contains the proofs of all the results in Section \ref{sec:convergence}.

\subsection{Hierarchical Models for Sparsity}
\begin{proof}[Proof of Theorem \ref{thm:deterministic error bound IAS}]
    The result follows from applying Theorem \ref{thm:reconstruction error theorem}. From Lemma \ref{lemma:decomposabileybityIASregularizer} we have that approximate decomposability holds with $c_1(\eta)=\frac{\eta}{2}$ and $c_2(\eta)=\frac{d\eta}{\sqrt{2}}(2-\log\eta)$. For the model subspace $\M(S)$ defined in \eqref{eq:sparse vector model subspace}, the subspace Lipschitz constant is given by
        \begin{align*}
            \Psi\bigl(\M(S)\bigr)=\sup_{u\in \M(S)}\frac{\|u\|_1}{\|u\|_2}=\sqrt{s}.
        \end{align*}
        The dual of the one norm is the infinity norm, so the assumption that $\lambda\geq \frac{2\|A^T\eps\|_{\infty}}{n}$ satisfies the condition of Theorem \ref{thm:reconstruction error theorem}.
        Since we have assumed $u^\star\in \M(S)$, we have that $u^\star_{\M^{\perp}}=0,$ and thus $\mathsf{R}(u^\star_{\M^{\perp}})=0$. Plugging each of these pieces into the bound \eqref{eq:generalbound} gives the desired result.
\end{proof}
\begin{proof}[Proof of Corollary \ref{thm:Gaussian Corollary}]
        First we verify that the given choice of $\lambda$ satisfies the assumption on Theorem \ref{thm:deterministic error bound IAS} with high probability. It is shown in \cite[Corollary 2]{negahban2012unified} that under the column normalization assumption, for any $\eps$ that is (sub-)Gaussian with identity covariance, it holds that
        \begin{align*}
            \mathbb{P}\left(\frac{\|A^T\eps\|_{\infty}}{n} \geq t\right)\leq 2 \exp \left(-\frac{nt^2}{2}+\log d \right).
        \end{align*}
        Taking $t^2=\frac{4\log d}{n},$ we conclude that the choice of $\lambda=4\sqrt{\frac{\log d}{n}}$ 
        satisfies the assumption of Theorem \ref{thm:deterministic error bound IAS} with probability $1-c_1\exp(-c_2n\lambda^2).$ Next we characterize the curvature and tolerance in the RSC condition for a Gaussian forward map. By \cite[Theorem 7.16]{wainwright2019high},  for any matrix $A\in \R^{n\times d}$ with rows i.i.d. from $\Nc(0,\Sigma)$ there exist positive constants $c_1<1<c_2$ such that, with probability at least $1-\frac{e^{-n/32}}{1-e^{-n/32}},$
        \begin{align*}
            \frac{\|A\Delta\|_2^2}{2n}\geq c_1\lambda_{\min}(\Sigma)\|\Delta\|_2^2-c_2\rho(\Sigma)^2\frac{\log d}{n}\|\Delta\|_1^2, \qquad \forall \Delta\in \R^d,
        \end{align*}
         where $\rho(\Sigma)^2:=\max_{1\leq j\leq d}\Sigma_{jj}$. In other words, the RSC condition holds with curvature $\kappa=c_1\lambda_{\min}(\Sigma)$ and tolerance $\tau^2=c_2\rho(\Sigma)^2\frac{\log d}{n}$. Consequently, to guarantee that $\tau \sqrt{s}\leq \frac{\sqrt{\kappa}}{32}$ as assumed in Theorem \ref{thm:deterministic error bound IAS}, we require that
    \begin{align*}
        \sqrt{n}\geq 32\sqrt{\frac{\frac{c_1}{c_2}\rho(\Sigma)^2s\log d}{\lambda_{\min}(\Sigma)}}.
    \end{align*}
    Given this requirement on $n$ we have satisfied the hypotheses of Theorem \ref{thm:deterministic error bound IAS}. Consequently, we plug in our characterizations of $\lambda$, $\kappa$, and $\tau$ into \eqref{eq:reconstruction error bound} and the desired result follows.
    \end{proof}

    \begin{proof}[Proof of Theorem \ref{thm:deterministic bound IAS LQ}]
        From Lemma \ref{lemma:decomposabileybityIASregularizer} we have that approximate decomposability holds with $c_1(\eta)=\frac{\eta}{2}$ and $c_2(\eta)=\frac{d\eta}{\sqrt{2}}(2-\log\eta)$. For the model subspace $\M(S_\delta)$ where $S_\delta$ is given by \eqref{eq:S delta}, the subspace Lipschitz constant is then given by
        \begin{align*}
            \Psi\bigl(\M(S_\delta)\bigr)=|S_\delta|^{1/2}.
        \end{align*}
        As such, we need to characterize the cardinality of $S_\delta$. We remark that 
        \begin{align*}
         R_q\geq \sum_{j=1}^d|u_j^\star|^q\geq \sum_{j\in S_\delta}|u_j^\star|^q\geq \delta^q|S_\delta|.
     \end{align*}
    It follows that
     \begin{align*}
         |S_\delta|\leq R_q\delta^{-q}.
     \end{align*}
     Under $\ell^q$ sparsity, we do not have that $u^\star_{S_\delta^C}=0$. However, we have the bound 
     \begin{align*}
         \|u^\star_{S_\delta^C}\|_1=\sum_{j\in S^c_{\delta}}|u_j^\star|^q|u_j^\star|^{1-q}\leq R_q\delta^{1-q},
     \end{align*}
     since $|u_j^\star|\leq\delta$ for all $j\in S_\delta^C$. Finally, since $q<1$, we have that
     \begin{align*}
         \|u^\star_{S_\delta}\|_1\leq R_q^{1/q}.
     \end{align*}
     Hence, we apply Theorem \ref{thm:reconstruction error theorem} and the desired result follows.
    \end{proof}
    \begin{proof}[Proof of Corollary \ref{thm:Gaussian Corollary IAS LQ}]
    Exactly as in the proof of Corollary \ref{thm:Gaussian Corollary}, the choice of $\lambda=4\sqrt{\frac{\log d}{n}}$ satisfies the assumption that $\lambda \geq \frac{2\|A^T\eps\|_{\infty}}{n}$ with probability $1-c_1\exp(-c_2n\lambda^2).$
     Again, exactly as in Corollary \ref{thm:Gaussian Corollary}, the RSC condition holds with curvature $\kappa=c_1\lambda_{\min}(\Sigma)$ and tolerance $\tau^2=c_2\rho(\Sigma)^2\frac{\log d}{n}$  with probability at least $1-\frac{e^{-n/32}}{1-e^{-n/32}}$. We pick $\delta$ in the definition of $S_{\delta}$ to be $\delta=\frac{\lambda}{\lambda_{\min}(\Sigma)}$ and plugging each of the above equations into Theorem \ref{thm:deterministic bound IAS LQ} yields the desired result.
\end{proof}

\subsection{Hierarchical Models for Group Sparsity}
\begin{proof}[Proof of Theorem \ref{thm: deterministic error bound GSIAS}]
    From Lemma \ref{lemma:decomposabilityGSIASregularizer} we have that approximate decomposability holds with $c_1(\eta)=\frac{\eta}{2}$ and $c_2(\eta)=\frac{k\eta}{\sqrt{2}}(2-\log \eta).$ For the model subspace $\M(S_{\G,\delta})$, where $S_{\G,\delta}$ is given in \eqref{eq: S G delta}, the subspace Lipschitz constant is given by
    \begin{align}\label{eq: GS Subspace Constant}
        \Psi \bigl(\M(S_{\G,\delta})\bigr)=\sup_{u\in \M(S_{\G,\delta})\backslash \{0\}}\frac{\sum_{j=1}^k\|u_{g_j}\|_{C_j}}{\|u\|_2}=\frac{1}{\sqrt{|S_{\G,\delta}|}}\sum_{j\in S_{\G,j}}\sqrt{\frac{1}{\lambda_{\min}(C_j)}}.
    \end{align}
    Note that if $C_j=I$ for all $j$, this reduces to $\Psi\bigl(\M(S_{\G,\delta})\bigr)=\sqrt{|S_{\G,\delta}|}.$
   We can upper bound the subspace Lipschitz constant as
    \begin{align}\label{eq:GS Subspace Constant Upper Bound}
        \Psi \bigl(\M(S_{\G,\delta})\bigr)\leq \sqrt{\frac{1}{\lambda_{\min}(C_{j^{*}} )}\left| S_{\G,\delta}\right|},
    \end{align}
    where $j^{*}=\argmax_{j\in S_{\G,\delta}}\frac{1}{\lambda_{\min} (C_j)}.$ If the matrices $C_j$ are all equal, then the upper bound \eqref{eq:GS Subspace Constant Upper Bound} is an equality. We will use this upper bound in what follows, but we remark that if the spectrum of the matrices $C_j$ varies drastically between groups of variables, one may be able to derive sharper bounds by working with \eqref{eq: GS Subspace Constant} directly. The dual norm of our group sparse norm is given by
    \begin{align*}
        \mathsf{R}^*(u):=\sup_{v\neq 0}\frac{\langle u,v\rangle}{\sum_{j=1}^k\|v_{g_j}\|_{C_j}}=\max_{g_j\in \G}\frac{2}{n}\|u_{g_j}\|_{C_j^{-1}},
    \end{align*}
    and we see that our assumption on $\lambda$ satisfies the condition of Theorem \ref{thm:reconstruction error theorem}. As in the proof of Theorem \ref{thm:deterministic bound IAS LQ}, one can show that
    \begin{align*}
        \left|S_{\G,\delta} \right| &\leq R_q\delta^{-q}, \\
         \mathsf{R}\bigl(u^\star_{\M(S_{\G,\delta}^C)}\bigr)&=\sum_{j\in S_{\G,\delta}^C}\|u^\star_{g_j}\|_{C_j}\leq R_q \delta^{1-q}, \\
          \mathsf{R}\bigl(u^\star_{\M(S_{\G,\delta})}\bigr)&=\sum_{j\in \M(S_{\G,\delta})}\|u^\star_{g_j}\|_{C_j}\leq R_q^{\frac{1}{q}}.
    \end{align*}
    These bounds along with Theorem \ref{thm:reconstruction error theorem} give the desired result.
\end{proof}
\begin{proof}[Proof of Corollary \ref{thm: Gaussian Corollary GS}]
    We first want to verify that the choice of $\lambda=2\left( \sqrt{\frac{p_{\max}^{}}{n}}+\sqrt{\frac{\log k}{n}}\right)$ satisfies
    \begin{align}\label{eq:lambdabound}
        \lambda \geq \max_{g_j\in \G} \frac{2}{n}\|(A^T\eps)_{g_j}\|_{2}
    \end{align}
    with high probability. We claim that if $A$ satisfies the block-normalization condition \eqref{eq:block normalization} and $\eps\sim \Nc(0,I)$, then we have
    \begin{align}\label{eq:claim}
        \mathbb{P}\left[ \frac{1}{n}\max_{1\leq j\leq k} \|X_{g_j}^T\eps\|_{2}\geq 2\left( \sqrt{\frac{p_{\max}}{n}}+\sqrt{\frac{\log k}{n}}\right)\right]\leq 2\exp\left( 2-\log k\right).
    \end{align}
       For a fixed group $g_j\in \G$ of size $p_j$, we consider the submatrix $A_{g_j}\in \R^{n\times p_j}$. We characterize the deviation above the mean using the same argument as \cite{negahban2012supplement} to get that, for all $t>0,$ 
    \begin{align}\label{eq:concentration GS Gaussian lemma}
    \mathbb{P}\left[ \frac{1}{n}\| X^T_{g_j}\eps \|_{2} \geq \mathbb{E}\left[ \frac{1}{n} \|X^T_{g_j}\eps\|_{2} \right] +t\right]\leq 2\exp \left( -\frac{n t^2}{2}\right).
    \end{align}
    The expectation in this bound can be bounded as follows
\begin{align}\label{eq:boundaux}
    \mathbb{E}\left[ \frac{1}{n}\| X^T_{g_j}\eps\|_{2} \right] \overset{\text{(i)}}{\leq} \sqrt{\frac{2}{n}}\mathbb{E}\left[\sqrt{\eps^T\eps}\right] \overset{\text{(ii)}}{\leq} \sqrt{\frac{2p_{\max}}{n}}, 
\end{align}
where for (i) we use the same Gaussian comparison principle argument as \cite{negahban2012supplement} and for (ii) we use Jensen's inequality and that $\eps\sim \Nc(0,I).$
Combining \eqref{eq:concentration GS Gaussian lemma} and \eqref{eq:boundaux} gives that, for all $t>0,$
\begin{align*}
    \mathbb{P}\left[ \frac{1}{n} \|X^T_{g_j}\eps\|_{2}\geq \sqrt{\frac{2p_{\max}}{n}} +t\right]\leq 2\exp \left( -\frac{n t^2}{2}\right).
\end{align*}
Applying a union bound over all $j\in \{1,\ldots , k\}$ and setting $t^2=\frac{4\log k}{n}$, we deduce that \eqref{eq:claim} holds as claimed. Hence, our choice of $\lambda$ satisfies the bound \eqref{eq:lambdabound} with probability at least $1-2/k^2$.
    From \cite[Proposition 1]{negahban2012supplement}, we conclude that the RSC condition holds with curvature $\kappa=\frac{1}{4}\lambda_{\min}(\Sigma)$ and tolerance $\tau^2=9\max_{g_j\in \G}\|(\Sigma^{1/2}_{g_j}\|_{op}\left(\sqrt{\frac{p_{\max}^{}}{n}}+\sqrt{\frac{\log k}{n}} \right)^2$.
    Applying Theorem \ref{thm: deterministic error bound GSIAS} with $\delta \asymp\lambda$ gives the desired result. 
\end{proof}

\subsection{Hierarchical Models for Sparse Representations}
\begin{proof}[Proof of Theorem \ref{thm: deterministic error bound OSIAS}]
    From Lemma \ref{lemma:decomposabilityOIASregularizer} we have that approximate decomposability holds with $c_1(\eta)=\frac{\eta}{2}$ and $c_2(\eta)=\frac{k\eta}{\sqrt{2}}(2-\log \eta).$ For the model subspace $\M(S_{W,\delta})$, where $S_{\G,\delta}$ is given in \eqref{eq: model subspace S W delta}, we have that the subspace Lipschitz constant is given by
    \begin{align*}
        \Psi \bigl(\M(S_{W,\delta})\bigr)=\sup_{u\in \M(S_{W,\delta}) \backslash \{0\}}\frac{\|W^T u\|_1}{\|u\|_2}=\sqrt{|S_{W,\delta}|}.
    \end{align*}
    The dual of $\|W^T\cdot\|_1$ is given by $\|W^T\cdot\|_{\infty}$, so our assumption on $\lambda$ satisfies the condition of Theorem \ref{thm:reconstruction error theorem}.
    By the same arguments as in the previous settings, we get the bounds
    \begin{align*}
        |S_{W,\delta}|&\leq R_q\delta^{-q}, \\
        R(u^\star_{\M(S_{W,\delta}^C)})&=\|W^Tu^\star_{\M(S_{W,\delta}^C)}\|_1\leq R_q\delta^{1-q}, \\
        R(u^\star_{\M(S_{W,\delta})})&=\|u^\star_{\M(S_{W,\delta})}\|_1\leq R_q^{1/q}.
    \end{align*}
    These bounds along with Theorem \ref{thm:reconstruction error theorem} give the desired result.
\end{proof}
\begin{proof}[Proof of Corollary \ref{thm:Gaussian Corollary SR}]
    By the normalization assumption and the fact that $\eps$ is Gaussian, we have that, for each $j=1,\ldots, k$,
    \begin{align*}
        \mathbb{P}\Bigl[ \bigl| \langle (AW)_j,\eps\rangle\bigr|\geq t\Bigr]\leq 2\exp\left( -\frac{nt^2}{2}\right).
    \end{align*}
    A union bound then gives that 
    \begin{align*}
         \mathbb{P}\left[ \frac{\|W^TA^T\eps\|_{\infty}}{n}\geq t\right]\leq 2\exp\left( -\frac{nt^2}{2}\right),
    \end{align*}
    and taking $t^2=\frac{4\log k}{n}$ we see that the choice of $\lambda=4\sqrt{\frac{\log k}{n}}$ satisfies the assumption of Theorem \ref{thm: deterministic error bound OSIAS}. The proof of \cite[Theorem 1]{raskutti2010restricted} can be adapted to show that, with probability at least $1-\frac{e^{-n/32}}{1-e^{-n/32}},$
    \begin{align*}
        \frac{\|A\Delta\|_2}{\sqrt{n}}\geq \frac{1}{4}\lambda_{\min}(\Sigma)-9\rho(W^T \Sigma W)^2\frac{\log k}{n}\|W^T\Delta\|_1^2, \qquad  \forall \Delta\in \R^d,
    \end{align*}
     where $\rho(W^T\Sigma W)^2:=\max_{1\leq j\leq k}(W^T\Sigma W)_{jj}$. Hence,  the RSC condition holds with curvature $\kappa=\frac{1}{4}\lambda_{\min}(\Sigma)$ and tolerance $\tau^2=9\rho(W^T\Sigma W)^2\frac{\log k}{n}.$ Picking $\delta \asymp \lambda$ and applying Theorem \ref{thm: deterministic error bound OSIAS} yields the desired result.
\end{proof}
\section{Conclusions}\label{sec:conclusions}
This paper has put forth a unified perspective on sparsity-promoting conditionally Gaussian hypermodels as M-estimators, bridging the fields of inverse problems and high-dimensional statistics. We have introduced the notion of approximate decomposability, which quantifies how close the regularization term resulting from a Bayesian hierarchical model is to a decomposable regularizer. We have then generalized the theory in \cite{negahban2012unified} to the setting of approximately decomposable regularizers, and proved the first known reconstruction error bounds for MAP estimators of conditionally Gaussian hypermodels. In particular, we have obtained bounds for widely-used hierarchical models to reconstruct sparse and group-sparse parameters, as well as for a novel model which promotes sparse representations. Our bounds suggest that the minimax rate can be achieved provided that the hyperpameter $\eta$ is taken sufficiently small. 

This work opens a number of important questions for future research. As noted above, from a statistical perspective, the hyperparameter $\eta$ should be taken as small as possible to promote sparsity. However, from a computational perspective, the positivity of $\eta$ enforces the strict convexity of the objective function, and taking $\eta$ to be too small may degrade the convergence of the IAS algorithm used to compute the MAP estimators. As such, an important avenue for research is to investigate the statistical-computational tradeoff in the choice of $\eta$, as well as computational comparisons between the Bayesian hypermodels and frequentist M-estimators that also admit efficient and well-studied optimization algorithms such as ISTA \cite{donoho1995noising}, FISTA \cite{beck2009fast}, and ADMM \cite{boyd2011distributed}.

The connection between conditionally Gaussian hypermodels and M-estimators can be leveraged to motivate the design of new hierarchical
Bayesian models and algorithms inspired by M-estimator theory. We conjecture that hierarchical models can be developed for tasks such as low-rank matrix estimation, sparse generalized linear models,
and sparsity-promoting dictionary learning, and that these models will inherit the same convenient computational and statistical properties as models considered in this paper. Likewise, our theory could be extended to the broader family of generalized gamma hyperpriors, for which greedy algorithms for efficient MAP estimation have been developed in \cite{calvetti2020sparse,calvetti2020sparsity}.

Finally, the analysis in this work relies heavily on the forward map satisfying the RSC condition \eqref{eq:RSC}. This condition holds for Gaussian forward maps, which may not be a realistic assumption for many inverse problems of interest. As such, verifying the RSC condition (possibly after preconditioning) for important forward maps in inverse problems, such as deconvolution or tomography \cite{kaipio2006statistical}, is an important direction for further research.


\section*{Acknowledgments}
DSA is grateful for the support of the NSF CAREER award DMS-2237628, the DOE grant DE-SC0022232, and the BBVA Foundation. The authors are thankful to Jiajun Bao and Jiaheng Chen for helpful feedback on a previous version of this manuscript, and to Omar Al-Ghattas for helpful discussions.

\bibliographystyle{siam} 
\bibliography{references}

\appendix

\section{Computational Properties of GS-IAS and O-IAS}\label{appendix:computationalproperties}

\begin{proof}[Proof of Theorem \ref{thm:GSIAS computational properties}] 
Lemma 2.2 in \cite{calvetti2015hierarchical} shows that the objective function \eqref{eq:GS functional} is strictly convex over $\R^d\times \R^k_+$ in the special case where each group is of size $p_j=3$. The proof holds as written for non-overlapping groups of arbitrary sizes. We denote the Hessian of the objective function \eqref{eq:GS functional} as
\begin{align*}
    H(u,\theta)=\begin{bmatrix}
        \nabla^2_{uu}\J(u,\theta) & \nabla^2_{u\theta} \J(u,\theta)\\
        \nabla^2_{\theta u} \J(u,\theta) & \nabla^2_{\theta \theta} \J(u,\theta)\\
    \end{bmatrix}.
\end{align*}
Let $e^\ell := u^\ell - \hat{u}.$ To show linear convergence, we apply Lemma 4.1 in \cite{calvetti2019hierachical}, which gives
\begin{align*}
    e^{\ell+1}&=\nabla^2_{uu}\J(\hat{u},\hat{\theta})^{-1}\nabla^2_{u\theta} \J(\hat{u},\hat{\theta})\nabla^2_{\theta \theta} \J(\hat{u},\hat{\theta})^{-1}\nabla^2_{\theta u} \J(\hat{u},\hat{\theta})e^\ell +o(\|e^\ell\|^2), 
\end{align*}
where each of the blocks of the Hessian are evaluated at the MAP estimator $(\hat{u},\hat{\theta}).$ Denoting by $\text{diag}(B_j)$ the block-diagonal matrix with blocks $B_j, 1 \le j \le k,$ we thus obtain that
\begin{align*}
     e^{\ell+1} &=\left(\frac{A^TA}{n}+\lambda D_{\hat{\theta}}^{-1} \right)^{-1} \text{diag}(\lambda\hat{\theta}_j C_j^{-1}\hat{u}_{g_j}) \times\\
     & \hspace{2cm} \times 
     \text{diag}(\lambda\hat{\theta}_j^{-3} \|\hat{u}_{g_j}\|_{C_j}^2 + \hat{\theta}_j^{-2}\eta)^{-1 }\text{diag}(\lambda\hat{\theta}_j C_j^{-1}\hat{u}_{g_j})e^\ell  +o(\|e^\ell\|^2) \\
    &  =D_{\hat{\theta}}^{1/2}\left(\frac{D_{\hat{\theta}}^{1/2}A^TAD_{\hat{\theta}}^{1/2}}{n\lambda}+I \right)^{-1} 
     \text{diag}(\hat{\theta}_j^{-3/2} C_j^{-1/2}\hat{u}_{g_j}) \times \\ 
     & \hspace{2cm} \times \text{diag}(\hat{\theta}_j^{-3} \|\hat{u}_{g_j}\|_{C_j}^2 + \hat{\theta}_j^{-2}\eta)^{-1 }\text{diag}(\hat{\theta}_j C_j^{-1}\hat{u}_{g_j})e^\ell  +o(\|e^\ell\|^2).
\end{align*}
Multiplying both sides by $D_{\theta}^{-1/2},$ taking norms, and neglecting higher order terms, we get that
\begin{align*}
    \|e^{\ell+1}\|_{D_{\hat{\theta}}}\leq \left\| \text{diag}\biggl(\frac{C_j^{1/2}\hat{u}_{g_j}\hat{u}_{g_j}^TC_j^{1/2}}{\|\hat{u}_{g_j}\|_{C_j}^2+\hat{\theta}_j\eta}\biggr)
    \right\|_{op} \|e^\ell\|_{D_{\hat{\theta}}}.
\end{align*}
Since for every $1\leq j\leq k$ it holds that $\|C_j^{1/2}\hat{u}_{g_j}\hat{u}_{g_j}^TC_j^{1/2}\|_{op}\leq \|\hat{u}_{g_j}\|_{C_j}^2$ and that $\hat{\theta}_j\eta>0$, we have linear convergence in the Mahalanobis norm. 

To show that $\hat{u}$ is the global minimizer of \eqref{eq: GS objective}, we refer to Theorem 2.1 in \cite{calvetti2015hierarchical}, which proves the result for the special case of each group having size $p_j=3$. Their proof still holds in our more general setting.
\end{proof}

\begin{proof}[Proof of Theorem \ref{thm:OIAS computational properties}]
    We first show that the O-IAS objective function $\mathsf{J}(z)$, where $z=(u,\theta)\in \R^{d}\times \R^k_+,$ given in \eqref{eq:O functional} is strictly convex over $\R^d\times \R^k_+$. We let $\tilde{\mathsf{J}}(\tilde{z})$ denote the IAS objective given in \eqref{eq:functional} over $\tilde{z}=(\tilde{u},\theta)\in \R^k\times \R^k_+$, with the forward map $\tilde{A}=AW$. From \cite{calvetti2019hierachical}, we know that $\tilde{\mathsf{J}}(\tilde{z})$ is strictly convex over $\R^k\times \R^k_+$. We then observe that $\mathsf{J}(z)=\tilde{\mathsf{J}}(C\tilde{z}),$ where
    \begin{align*}
        C=\begin{bmatrix}
            W^T & 0\\
            0 & I \\
        \end{bmatrix}\in \R^{2k\times d+k}.
    \end{align*}
    Since this matrix has trivial nullspace and strict convexity is preserved under such a transformation, we conclude that $\mathsf{J}(z)$ is strictly convex. Since $\mathsf{J}(z)$ tends to infinity as $\theta_j\to 0$, $\theta_j\to \infty$, and $\|u\|\to \infty$, a minimizer $(\hat{u},\hat{\theta})$ exists and is unique. As in the proof of Theorem \ref{thm:GSIAS computational properties}, we let $e^{\ell}=u^{\ell}-\hat{u}$ and apply Lemma 4.1 in \cite{calvetti2019hierachical}, which gives 
    \begin{align*}
    e^{\ell+1}&=\nabla^2_{uu}\J(\hat{u},\hat{\theta})^{-1}\nabla^2_{u\theta} \J(\hat{u},\hat{\theta})\nabla^2_{\theta \theta} \J(\hat{u},\hat{\theta})^{-1}\nabla^2_{\theta u} \J(\hat{u},\hat{\theta})e^\ell +o(\|e^\ell\|^2), 
\end{align*}
where each of the blocks of the Hessian are evaluated at the MAP estimator $(\hat{u},\hat{\theta}).$ Plugging in the expressions of the Hessian blocks, we get that
\begin{align*}
    e^{\ell+1}&=\left(\frac{A^TA}{n}+\lambda W D_{\hat{\theta}}^{-1} W^T  \right)^{-1}W
    \times \\
    & \hspace{1cm} \times \text{diag}(\lambda W^T \hat{u}/\hat{\theta}^2)\left( \text{diag}(\lambda(W^T\hat{u})^2/\hat{\theta}^3+\eta/\hat{\theta}^2)\right)^{-1}\text{diag}(\lambda W^T\hat{u}/\hat{\theta}^2)W^Te^{\ell}+o(\|e^\ell\|^2).
\end{align*}
Multiplying both sides by $\text{diag}(1/\hat{\theta}^{1/2})W^T,$ we get that
\begin{align*}
    \text{diag}(\hat{\theta}^{-1/2})W^Te^{\ell+1}& =\text{diag}(\hat{\theta}^{-1/2})W^T\left(\frac{A^TA}{n}+\lambda W D_{\hat{\theta}}^{-1} W^T  \right)^{-1}W \times \\
    &\hspace{2cm} \times
    \text{diag}\left( \frac{\lambda(W^T\hat{u})^2}{(W^T\hat{u})^2\hat{\theta}^{1/2}+\eta\hat{\theta}^{3/2}}\right)\text{diag}(\hat{\theta}^{-1/2})W^Te^{\ell} +o(\|e^\ell\|^2).
\end{align*}
Taking norms of both sides, we see that if
\begin{align*}
    \mu= \left\|\text{diag}(\hat{\theta}^{-1/2})W^T\left(\frac{A^TA}{n}+\lambda W D_{\hat{\theta}}^{-1} W^T  \right)^{-1}W\text{diag}\left( \frac{\lambda(W^T\hat{u})^2}{(W^T\hat{u})^2\hat{\theta}^{1/2}+\eta\hat{\theta}^{3/2}}\right) \right\|<1,
\end{align*}
then we have linear convergence in the Mahalanobis norm. Letting $W\text{diag}(1/\hat{\theta})W^T=ZZ$, where $Z\in \R^{d\times d}$ is the unique matrix square root of $W\text{diag}(1/\hat{\theta})W^T$, we have that 
 \begin{align*}
     \mu &= \left\|\text{diag}(1/\hat{\theta}^{1/2})W^TZ^{-1}\left(\frac{Z^{-1}A^TAZ^{-1}}{\lambda n}+I\right)^{-1}Z^{-1}W\text{diag}\left( \frac{(W^T\hat{u})^2}{(W^T\hat{u})^2\hat{\theta}^{1/2}+\eta\hat{\theta}^{3/2}}\right) \right\| \\
     &\leq \left\|\text{diag}(1/\hat{\theta}^{1/2})W^TZ^{-1}Z^{-1}W\text{diag}\left( \frac{(W^T\hat{u})^2}{(W^T\hat{u})^2\hat{\theta}^{1/2}+\eta\hat{\theta}^{3/2}}\right) \right\| \\
     &= \left\|\text{diag}(1/\hat{\theta}^{1/2})W^T\left(W\text{diag}(1/\hat{\theta})W^T\right)^{-1}W\text{diag}\left( \frac{(W^T\hat{u})^2}{(W^T\hat{u})^2\hat{\theta}^{1/2}+\eta\hat{\theta}^{3/2}}\right) \right\|.
 \end{align*}
 Since $WW^T=I$, the SVD of $W$ must have the form 
 \begin{align*}
     W=U\begin{bmatrix}
         I & 0
     \end{bmatrix}V^T=U\tilde{V}^T,
 \end{align*}
 where $\tilde{V}\in \R^{k\times d}$ is the matrix containing the first $d$ columns of $V$. Plugging this expression into our upper bound for $\mu$, we get that
 \begin{align*}
     \mu &\leq \left\|\text{diag}(1/\hat{\theta}^{1/2})\tilde{V}U^T\left(U\tilde{V}^T\text{diag}(1/\hat{\theta})\tilde{V}U^T\right)^{-1}U\tilde{V}^T\text{diag}\left( \frac{(W^T\hat{u})^2}{(W^T\hat{u})^2\hat{\theta}^{1/2}+\eta\hat{\theta}^{3/2}}\right) \right\| \\
     &= \left\|\text{diag}(1/\hat{\theta}^{1/2})\tilde{V}\left(\tilde{V}^T\text{diag}(1/\hat{\theta})\tilde{V}\right)^{-1}\tilde{V}^T\text{diag}\left( \frac{(W^T\hat{u})^2}{(W^T\hat{u})^2\hat{\theta}^{1/2}+\eta\hat{\theta}^{3/2}}\right) \right\|.
 \end{align*}
 Making the substitution $V'=\text{diag}(1/\hat{\theta}^{1/2})$, we get that
 \begin{align*}
     \mu \leq \left\|V'\left(V'^TV' \right)^{-1}V'^T\text{diag}\left( \frac{(W^T\hat{u})^2}{(W^T\hat{u})^2+\eta\hat{\theta}}\right) \right\|.
 \end{align*}
 Observing that $V'\left(V'^TV' \right)^{-1}V'^T$ is an orthogonal projector onto the range of $V'$, we have that
 \begin{align*}
     \mu \leq \left\|\text{diag}\left( \frac{(W^T\hat{u})^2}{(W^T\hat{u})^2+\eta\hat{\theta}}\right) \right\|<1.
 \end{align*}
 Consequently, we have shown linear convergence of the O-IAS algorithm in the Mahalanobis norm.\\

 We will now show that the minimizer of \eqref{eq:O functional} over $u$ and $\theta$ can be written as the minimizer of $\mathsf{J}\bigl(u,f(u)\bigr),$ where
 \begin{align*}
     f_j(u)=\left(\frac{\eta}{2}+\sqrt{\frac{\eta^2}{4}+\frac{(W^Tu)_j}{2}} \right).
 \end{align*}
Since the minimizer $(\hat{u},\hat{\theta})$ is a critical point, it follows that 
\begin{align}
    \nabla_{u}\mathsf{J}(\hat{u},\hat{\theta})&=\left(\frac{A^TA}{n}+\lambda W\text{diag}(1/\hat{\theta})W^T \right)\hat{u}-\frac{A^Ty}{n}=0, \nonumber \\ 
    \nabla_{\theta}\mathsf{J}(\hat{u},\hat{\theta})&=\sum_{j=1}^k\left(\frac{(W^T\hat{u})_j^2}{2\hat{\theta}_j^2}-\frac{\eta}{\hat{\theta}_j}-1 \right)e_j=0.  \label{eq:theta derivative}
\end{align}
Given $\hat{u}$, we solve \eqref{eq:theta derivative} for $\hat{\theta}$ in terms of $\hat{u}$. Some algebra shows that the solution is given by $\hat{\theta}=f(\hat{u})$, and thus $\hat{u}$ is a critical point of 
\begin{align}\label{eq:overcomplete just in x}
    \mathsf{F}(u)=\mathsf{J} \bigl(u,f(u)\bigr).
\end{align}
If $\tilde{u}$ is a critical point of \eqref{eq:overcomplete just in x}, letting $\tilde{\theta}=f(\tilde{u})$, we get that $(\tilde{u},\tilde{\theta})$ is a critical point of $\mathsf{J}(u,\theta)$ and by uniqueness of the minima we have that $\hat{u}=\tilde{u}$, completing the proof.
\end{proof}

\section{Approximate Decomposability: Examples}\label{appendix:exappdec}
    \begin{proof}[Proof of Lemma \ref{lemma:decomposabileybityIASregularizer}]
        By the triangle inequality, it holds that
        \begin{align}\label{eq:first inequalities}
            \frac{|u_j|}{\sqrt{2}}\leq f_j(u)\leq \eta +\frac{|u_j|}{\sqrt{2}}, \quad \text{and} \quad \eta \leq f_j(u).
        \end{align}
        Using these inequalities, we deduce that
        \begin{align*}
            \mathsf{R}_{\eta}(u)&=\frac{1}{2\sqrt{2}}\sum_{j=1}^d\left[\frac{u_j^2}{f_j(u)}+2f_j(u)-2\eta \log f_j(u) \right] \\
            &\leq \frac{1}{2\sqrt{2}}\sum_{j=1}^d\left[\sqrt{2}|u_j|+2\eta +\sqrt{2}|u_j|-2\eta \log \left(\eta\right) \right] \le \|u\|_1 + \frac{d\eta}{\sqrt{2}} (1 - \log \eta),
        \end{align*}
        proving the upper bound in \eqref{eq:IASdecombounds}. For the lower bound, we have
        \begin{align*}
  \mathsf{R}_\eta(u) &\overset{\text{(i)}}{\geq} \frac{1}{2\sqrt{2}}\sum_{j=1}^d\left[\frac{u_j^2}{\eta+\frac{|u_j|}{\sqrt{2}}}+\sqrt{2}|u_j|-2\eta \log \left(\eta+\frac{|u_j|}{\sqrt{2}} \right) \right] \\
  &\overset{\text{(ii)}}{\geq}  \frac{1}{2\sqrt{2}}\sum_{j=1}^d\left[\frac{u_j^2}{\eta+\frac{|u_j|}{\sqrt{2}}}+\sqrt{2}|u_j|-2\eta \frac{|u_j|}{\sqrt{2}} \right] \\
  &\overset{\text{(iii)}}{\geq}  \frac{1}{2\sqrt{2}}\sum_{j=1}^d\left[\sqrt{2}|u_j|-2\eta+\sqrt{2}|u_j|-\sqrt{2}\eta|u_j| \right]= \Bigl(1 - \frac{\eta}{2}\Bigr) \| u \|_1 - \frac{d\eta}{\sqrt{2}},
        \end{align*}
        where (i) follows by the inequalities in \eqref{eq:first inequalities}, (ii) follows using that for $0<\eta<1$ it holds that $\log\Bigl(\eta + \frac{|u_j|}{\sqrt{2}}\Bigr)\leq \frac{|u_j|}{\sqrt{2}},$ and (iii) follows using that  $\frac{u_j^2}{\eta+\frac{|u_j|}{\sqrt{2}}}\geq \sqrt{2}|u_j|-2\eta.$
    \end{proof}

\begin{proof}[Proof of Lemma \ref{lemma:decomposabilityGSIASregularizer}]
    Similar to the proof of Lemma \ref{lemma:decomposabileybityIASregularizer}, replacing $|u_j|$ with $\|u_{g_j}\|_{C_j}.$
\end{proof}

\begin{proof}[Proof of Lemma \ref{lemma:decomposabilityOIASregularizer}]
    The statement follows from Lemma \ref{lemma:decomposabileybityIASregularizer} applied to $W^Tu\in \R^k$.
\end{proof}

\section{Approximate Decomposability and General Bounds}\label{appendix:proofgeneralbounds}
In this appendix we prove the key property of approximately decomposable regularizers in Lemma \ref{lemma:decomposabilityproperty} and the general reconstruction bound in Theorem \ref{thm:reconstruction error theorem}.
To control the reconstruction error $\hat{\Delta} = \hat{u} - u^\star$ between the minimizer of the objective $\mathsf{F}$ in \eqref{eq:objective} and the true parameter $u^\star$ generating the data according to  \eqref{eq:frequentist inverse problem}, we introduce a function $\F:\R^d\to \R$ given by
\begin{align}\label{eq:F(Delta)}
\begin{split}
    \F(\Delta) :&=\mathsf{F}(u^\star + \Delta) - \mathsf{F}(u^\star) \\
    &= \frac{1}{2n} \Bigl[ \|y-A(u^\star+\Delta)\|_2^2-\|y-Au^\star\|_2^2 \Bigl] +\lambda \Bigl[ \mathsf{R}(u^\star+\Delta)-\mathsf{R}(u^\star)\Bigr].
\end{split}    
\end{align}
By construction, $\F(0)=0$. Moreover, for $\hat{\Delta}=\hat{u}-u^\star$ we have the \textit{basic inequality}
\begin{equation}\label{eq:basicinequality}
    \F(\hat{\Delta}) = \mathsf{F}(\hat{u}) - \mathsf{F}(u^\star) \le 0,
\end{equation}
since $\hat{u}$ is the minimizer of $\mathsf{F}.$  We will leverage this inequality to control the reconstruction error.

    \begin{proof}[Proof of Lemma \ref{lemma:decomposabilityproperty}]
        We first prove a lower bound on the difference of the regularizer evaluated at the minimizer and at the true parameter. Using approximate decomposability of $\mathsf{R}_\eta$, reverse triangle inequality, and decomposability of $\mathsf{R}$ over $\M,$ it follows that 
        \begin{align}\label{eq:R1} 
        \begin{split}
            \mathsf{R}_{\eta}(\hat{u})&\geq \left(1-c_1^L(\eta) \right)\mathsf{R}(u^\star+\hat{\Delta})-c_2^L(\eta) \\
            &\geq \left(1-c_1^L(\eta) \right)\left(\mathsf{R}(u^\star_\M+\hat{\Delta}_{\M^{\perp}})-\mathsf{R}(u^\star_{\M^\perp})-\mathsf{R}(\hat{\Delta}_\M)\right)-c_2^L(\eta) \\
            &\geq \left(1-c_1^L(\eta) \right)\left(\mathsf{R}(u^\star_\M)+\mathsf{R}(\hat{\Delta}_{\M^{\perp}})-\mathsf{R}(u^\star_{\M^\perp})-\mathsf{R}(\hat{\Delta}_\M)\right)-c_2^L(\eta).
         \end{split}   
        \end{align}
 Similarly, using approximate decomposability of $\mathsf{R}_\eta$ and decomposability of $\mathsf{R},$ we obtain       
\begin{equation}\label{eq:R2}
    \mathsf{R}_{\eta}(u^\star) \le \left(1 + c_1^U(\eta)\right) \Bigl( \mathsf{R}(u_\M^\star) + \mathsf{R}(u^\star_{\M^\perp}) \Bigr) + c_2^U(\eta). 
\end{equation}       
        Combining \eqref{eq:R1} and \eqref{eq:R2} gives         
        \begin{align}\label{eq:regularizer deviation bound}
        \begin{split}
            \mathsf{R}_\eta(\hat{u})-\mathsf{R}_{\eta}(u^\star)
            &\geq \left(1-c_1^L(\eta) \right)\left(\mathsf{R}(\hat{\Delta}_{\M^\perp})-\mathsf{R}(\hat{\Delta}_\M) \right)
            -c_1(\eta)\mathsf{R}(u^\star_{\M})
             -2\mathsf{R}(u^\star_{\M^\perp})-c_2(\eta),
        \end{split}
        \end{align}
        where $c_i(\eta) = c_i^L(\eta) + c_i^U(\eta)$ and we used that $c_1(\eta) \mathsf{R}(u^\star_{\M^\perp}) \ge 0.$
        We next show a lower bound on the difference of the likelihood evaluated at the minimizer and at the true parameter. By direct calculation, it holds that
        \begin{align*}
            \frac{1}{2n} \Bigl[\|y-A\hat{u}\|_2^2- \|y-Au^\star\|_2^2 \Bigr]=\frac{1}{2n}\|A\hat{\Delta}\|_2^2-\frac{1}{n}\langle A^T\eps,\hat{\Delta}\rangle\geq -\frac{1}{n}\left|\langle A^T\eps,\hat{\Delta}\rangle\right|.
        \end{align*}
        Then, using H\"older's inequality with the regularizer and its dual, we get that
        \begin{align}\label{eq:likelihood deviation bound}
        \begin{split}
            \frac{1}{2n} \Bigl[\|y-A\hat{u}\|_2^2- \|y-Au^\star\|_2^2 \Bigr] &\geq -\frac{1}{n}\mathsf{R}^*(A^T\eps)\mathsf{R}(\hat{\Delta})=-\frac{1}{n}\mathsf{R}^*(A^T\eps)\left(\mathsf{R}(\hat{\Delta}_\M)+\mathsf{R}(\hat{\Delta}_{\M^\perp}) \right) \\
           & \geq -\frac{\lambda}{2}\left(\mathsf{R}(\hat{\Delta}_\M)+\mathsf{R}(\hat{\Delta}_{\M^\perp}) \right),
        \end{split}   
        \end{align}
        where the last inequality follows by our assumption on $\lambda.$ Finally, \eqref{eq:regularizer deviation bound} and \eqref{eq:likelihood deviation bound} combined with the basic inequality \eqref{eq:basicinequality} give
        \begin{align*}
            0 \geq \F(\hat{\Delta}) 
            &\geq \lambda 
            \biggl( \left(1-c_1^L(\eta) \right)\left(\mathsf{R}(\hat{\Delta}_{\M^\perp})-\mathsf{R}(\hat{\Delta}_\M) \right)-c_1(\eta)\mathsf{R}(u^\star_{\M}) 
            -2\mathsf{R}(u^\star_{\M^\perp})-c_2(\eta)\biggl)  \\
            &-\frac{\lambda}{2}\left( \mathsf{R}(\hat{\Delta}_\M)+\mathsf{R}(\hat{\Delta}_{\M^\perp})\right),
        \end{align*}
        which after rearranging implies that
        \begin{align*}
            \left(\frac{1}{2}-c_1^L(\eta) \right)\mathsf{R}(\hat{\Delta}_{\M^\perp})\leq & \left( \frac{3}{2}+c_1^L(\eta)\right)\mathsf{R}(\hat{\Delta}_\M)+2\mathsf{R}(u^\star_{\M^\perp})+c_1(\eta)\mathsf{R}(u^\star_\M)+c_2(\eta).
        \end{align*}
        The desired result follows from the assumption that $c_1^L(\eta)<\frac{1}{4}.$
    \end{proof}

Letting $\delta>0$ be an error radius that we will specify later, we define the set 
\begin{align*}
    \mathbb{k}(\delta):=\mathbb{C}_{\eta}(\M)\cap \{\|\Delta\|_2=\delta\}.
\end{align*}
The following result from \cite[Lemma 4] {negahban2012supplement} demonstrates that it is sufficient to show a positive lower bound for the function $\F$ given in \eqref{eq:F(Delta)} over the set $\mathbb{K}(\delta).$
\begin{lemma}\label{lemma:radius lemma}
    If $\F(\Delta)>0$ for all vectors $\Delta\in \mathbb{K}(\delta)$, then $\|\hat{\Delta}\|_2\leq \delta.$
    \begin{proof}
        The proof in \cite{negahban2012supplement} still holds essentially as written in our slightly modified setting. In particular, the proof does not require that the regularizers in $\F(\Delta)$ be decomposable. The regularizers in our definition of $\mathbb{C}_{\eta}$ are decomposable with respect to $\M$, and thus the argument this set is ``star shaped'' (see \cite{negahban2012supplement} or \cite[Chapter 9]{wainwright2019high} for a definition and discussion of this property) holds as written. 
    \end{proof}
\end{lemma}
We are now ready to prove Theorem \ref{thm:reconstruction error theorem}.
\begin{proof}[Proof of Theorem \ref{thm:reconstruction error theorem}]
    Fix a radius $\delta>0$, which we will specify at the conclusion of the proof. By Lemma \ref{lemma:radius lemma}, we proceed by deriving a positive lower bound for $\F(\Delta)$ over $\mathbb{K}(\delta),$ which will yield a bound on $\|\hat{\Delta}\|_2^2.$ By definition of $\F$, the RSC condition, and the inequality \eqref{eq:regularizer deviation bound} 
    \begin{align*}
        \F(\Delta)&=\frac{1}{2n} \Bigl[ \|y-A(u^\star+\Delta)\|_2^2 - \|y-Au^\star\|_2^2 \Bigr] +\lambda \left[ \mathsf{R}(u^\star+\Delta)-\mathsf{R}(u^\star)\right]\\
        &\geq \langle A^T\eps,\Delta\rangle +\frac{\kappa}{2}\|\Delta\|_2^2-\tau^2\mathsf{R}^2(\Delta)+\lambda \left[ \mathsf{R}(u^\star+\Delta)-\mathsf{R}(u^\star)\right]\\
        &\geq\langle A^T\eps,\Delta\rangle +\frac{\kappa}{2}\|\Delta\|_2^2-\tau^2\mathsf{R}^2(\Delta)\\
        &+\lambda\Big[ \left(1-c_1^L(\eta) \right)\left(\mathsf{R}(\hat{\Delta}_{\M^\perp})-\mathsf{R}(\hat{\Delta}_\M) \right)-c_1(\eta)\mathsf{R}(u^\star_{\M})-2\mathsf{R}(u^\star_{\M^\perp})-c_2(\eta)\Big].
    \end{align*}
    Applying H\"older's inequality and using the assumption on $\lambda$, we have that $|\langle A^T\eps,\Delta\rangle|\leq \frac{\lambda}{2}\mathsf{R}(\Delta),$ which gives 
    \begin{align*}
        &\F(\Delta)\geq \frac{\kappa}{2}\|\Delta\|_2^2-\tau^2\mathsf{R}^2(\Delta)-\frac{\lambda}{2}\mathsf{R}(\Delta) \\
        &+\lambda\Big[ \left(1-c_1^L(\eta) \right)\left(\mathsf{R}(\hat{\Delta}_{\M^\perp})-\mathsf{R}(\hat{\Delta}_\M) \right)-c_1(\eta)\mathsf{R}(u^\star_{\M})
        -2\mathsf{R}(u^\star_{\M^\perp})-c_2(\eta)\Big]\\
        &= \frac{\kappa}{2}\|\Delta\|_2^2-\tau^2\mathsf{R}^2(\Delta) \\
        &+\lambda\biggl[ \left(1-c_1^L(\eta) -\frac{1}{2}\right)\mathsf{R}(\hat{\Delta}_{\M^\perp})-\left(1-c_1^L(\eta) +\frac{1}{2}\right)\mathsf{R}(\hat{\Delta}_\M)-c_1(\eta)\mathsf{R}(u^\star_{\M})
        -2\mathsf{R}(u^\star_{\M^\perp})-c_2(\eta)\biggr],
    \end{align*}
    where the equality follows from the decomposability of $\mathsf{R}.$ Notice that, for any $\Delta\in \mathbb{C}_{\eta}(\M),$ 
        \begin{align*}
        \mathsf{R}^2(\Delta)= 
        \bigl( \mathsf{R}(\Delta_\M) + \mathsf{R}(\Delta_{\M^\perp}) \bigr)^2
        &\leq 16\bigl( 2\Psi(\M)\|\Delta\|_2+2\mathsf{R}(u^\star_{\M^\perp})+c_1(\eta)\mathsf{R}(u^\star_\M)+c_2(\eta)\bigr)^2\\ &\leq 64 \bigl[4 \Psi^2(\M)\|\Delta\|_2^2+4\mathsf{R}^2(u^\star_{\M^\perp})+c_1^2(\eta)\mathsf{R}^2(u^\star_\M)+c_2^2(\eta)\bigr].
    \end{align*}
    Plugging this bound into the previous inequality, we get 
    \begin{align*}
        \F(\Delta)&\geq \frac{\kappa}{2}\|\Delta\|_2^2- 64\tau^2 \bigl[4 \Psi^2(\M)\|\Delta\|_2^2+4\mathsf{R}^2(u^\star_{\M^\perp})+c_1^2(\eta)\mathsf{R}^2(u^\star_\M)+c_2^2(\eta)\bigr]\\
       &-\lambda\left[ \left(1-c_1^L(\eta) +\frac{1}{2}\right)\mathsf{R}(\hat{\Delta}_\M)+c_1(\eta)\mathsf{R}(u^\star_{\M})
        +2\mathsf{R}(u^\star_{\M^\perp})+c_2(\eta)\right]\\
        &\geq \frac{\kappa}{2}\|\Delta\|_2^2- 64\tau^2 \bigl[4 \Psi^2(\M)\|\Delta\|_2^2+4\mathsf{R}^2(u^\star_{\M^\perp})+c_1^2(\eta)\mathsf{R}^2(u^\star_\M)+c_2^2(\eta)\bigr]\\
       &-\lambda\left[ \left(1-c_1^L(\eta) +\frac{1}{2}\right)\Psi(\M)\|\Delta\|_2+ c_1(\eta)\mathsf{R}(u^\star_{\M})
        +2\mathsf{R}(u^\star_{\M^\perp})+c_2(\eta)\right].
    \end{align*}
    By our assumption that $\tau \Psi(\M)\leq \frac{\sqrt{\kappa}}{32}$, it follows that
    \begin{align*}
        \F(\Delta)&\geq \frac{\kappa}{4}\|\Delta\|_2^2-64\tau^2\left[4\mathsf{R}^2(u^\star_{\M^\perp})+c_1^2(\eta)\mathsf{R}^2(u^\star_\M)+c_2^2(\eta)\right]\\
       &-\lambda\left[ \left(1-c_1^L(\eta) +\frac{1}{2}\right)\Psi(\M)\|\Delta\|_2+c_1(\eta)\mathsf{R}(u^\star_{\M})
        +2\mathsf{R}(u^\star_{\M^\perp})+c_2(\eta)\right].
    \end{align*}
    The right-hand side of this inequality is a quadratic polynomial in $\|\Delta\|_2$ that is positive for $\|\Delta\|$ sufficiently large. In particular, the quadratic formula shows that taking 
       \begin{align*}
        \delta&= 72\frac{\lambda^2}{\kappa^2}\Psi^2(\M)  \\
        &+\frac{8}{\kappa}\biggl(64\tau^2\left[4\mathsf{R}^2(u^\star_{\M^\perp})+c_1^2(\eta)\mathsf{R}^2(u^\star_\M)+c_2^2(\eta)\right]
        +\lambda\left[ c_1(\eta)\mathsf{R}(u^\star_\M)+ 2\mathsf{R}(u^\star_{\M^\perp}) + c_2(\eta) \right]\biggr)
    \end{align*}
    guarantees a positive lower bound on $\F(\Delta)$ over $\mathbb{K}(\delta)$, completing the proof.
\end{proof}

\end{document}